\documentclass[11pt]{amsart}

\usepackage{amsmath, amsthm, amssymb}
\usepackage{amsfonts}
\usepackage{mathtools}

\usepackage[cal=boondox]{mathalfa}
\usepackage{bbm}
\usepackage{mathrsfs}
\usepackage{stmaryrd}
\usepackage{esint}
\usepackage[T1]{fontenc}
\usepackage{lmodern}

\usepackage{graphicx}
\usepackage{tikz}
\usepackage{tikz-cd}
\usetikzlibrary{arrows}
\usepackage[all,cmtip]{xy}
\usepackage{amscd}
\usepackage{picinpar}

\usepackage{enumerate}
\usepackage{enumitem}
\usepackage{comment}
\usepackage{array}
\usepackage{extarrows}
\usepackage{verbatim}
\usepackage{calc}
\usepackage{xcolor}

\usepackage{hyperref}

\numberwithin{equation}{section}

\theoremstyle{plain}
\newtheorem{thm}{Theorem}
\numberwithin{thm}{subsection}
\newtheorem{cor}[thm]{Corollary}
\newtheorem{lem}[thm]{Lemma}
\newtheorem{prop}[thm]{Proposition}

\newtheorem{mainthm}{Theorem}

\theoremstyle{definition}

\theoremstyle{remark}
\newtheorem{rem}[thm]{Remark}

\DeclareMathOperator{\coh}{H}
\DeclareMathOperator{\Hilb}{Hilb}

\DeclareMathOperator{\pr}{pr}

\DeclareMathOperator{\Sym}{Sym}

\DeclareMathOperator{\Div}{Div}

\DeclareMathOperator{\Sing}{Sing}

\DeclareMathOperator{\Spec}{Spec}
\DeclareMathOperator{\length}{length}

\def\bA{{\mathbb{A}}}

\def\bC{{\mathbb{C}}}

\def\bP{{\mathbb{P}}}

\def\cB{{\mathscr{B}}}
\def\cC{{\mathscr{C}}}
\def\cD{{\mathscr{D}}}

\def\cF{{\mathscr{F}}}

\def\cN{{\mathscr{N}}}
\def\cO{{\mathscr{O}}}
\def\cP{{\mathscr{P}}}
\def\cQ{{\mathscr{Q}}}
\def\cR{{\mathscr{R}}}
\def\cS{{\mathscr{S}}}
\def\cT{{\mathscr{T}}}

\def\cU{{\mathscr{U}}}

\def\cX{{\mathscr{X}}}
\def\cY{{\mathscr{Y}}}
\def\cZ{{\mathscr{Z}}}

\def\aO{{\mathcal{O}}}
\def\aF{{\mathcal{F}}}
\def\aL{{\mathcal{L}}}

\def\aH{{\mathcal{H}}}

\begin{document}

\title[Hyperbolicity of general surfaces]{Kobayashi Hyperbolicity of General Surfaces via the Poincaré Problem}

\author{Song-Yan Xie}
\address{State Key Laboratory of Mathematical Sciences, Academy of Mathematics and Systems Science, Chinese Academy of Sciences, Beijing 100190, China;  School of Mathematical Sciences, University of Chinese Academy of Sciences, Beijing 100049, China.}
\email{xiesongyan@amss.ac.cn}

\author{Shengyuan Zhao}
\address{Universit\'e Paul Sabatier, Institut de Math\'ematiques de Toulouse, 118, route de Narbonne, F-31062 Toulouse, France}
\email{shengyuan.zhao@math.univ-toulouse.fr}


\begin{abstract}
We prove that a general surface in $\mathbb{P}^3$ of degree at least $18$ contains no rational or elliptic curves, strengthening the classical result of Clemens by replacing the original very general assumption by a genuine Zariski-open condition. Previously, nonexistence results in the ``general'' setting were known only in much higher degrees. Combining this with established algebraic degeneracy results for entire curves, we deduce the Kobayashi hyperbolicity of a general surface in
$\mathbb{P}^3$ of degree at least $18$, thereby resolving a question asked by Demailly--El Goul.

Our proof uses foliations induced by $2$-jet differentials. Two independent such differentials give rise to a multi-foliation tangent to all rational and elliptic curves. We establish a Poincar\'e-type bound for its algebraic leaves, yielding a mechanism to upgrade very general statements to general ones.

Our method also applies to complements of plane curves. In particular, we prove that the complement of two general cubic curves in $\mathbb{P}^2$ is hyperbolically embedded. 
\end{abstract}

\maketitle

\section{Introduction}

\subsection{Rational and elliptic curves on hypersurfaces}\label{subsection:1.1}
A foundational result of Clemens~\cite{Clemens1986} asserts that a
\emph{very general} hypersurface $X$ of degree $d$ in $\bP^n$ contains no rational curves if $d\geq 2n-1$, and no elliptic curves if
$d\geq 2n$. Here, ``very general'' means that $X$ lies outside a countable union of proper Zariski-closed subsets of the parameter space. In particular, when $n=3$, a very general surface in $\mathbb{P}^3$ contains no rational curves if its degree is at least $5$, and no elliptic curves if its degree is at least $6$. 
This bound was later sharpened by Xu~\cite{Xu1994}, who proved that a very general surface of degree $d \geq 5$ in $\mathbb{P}^3$ contains no elliptic curves. 

Clemens' approach is based on a delicate deformation-theoretic analysis of the normal bundle of a curve in a hypersurface; it is refined and extended by Ein, Voisin, Xu (see e.g.~\cite{Ein1988, Ein1991, Xu1994, Voisin1996}). Beyond its immediate scope, Clemens' theorem has had a lasting influence, notably inspiring  Siu's method of slanted vector fields in hyperbolicity problems~\cite{Siu2004, Paun2008, Rousseau2009, Merker2009, Darondeau2016}. 

We note that the terms ``general'', ``very general'', ``generic'', and ``very generic'' are used with inconsistent meanings across the literature. 
The objective of our work is to replace the very general condition required by Clemens and Xu by a genuine Zariski-open condition in the case of rational and elliptic curves on surfaces. In this paper we use ``general'' to refer to such a Zariski-open condition.

\subsection{The analytic aspect: Kobayashi hyperbolicity}
For a compact complex manifold $X$, Kobayashi hyperbolicity is characterized by Brody's lemma~\cite{Brody} as follows: $X$ is hyperbolic if and only if it contains no nonconstant entire holomorphic curve $\mathbb{C} \to X$. In particular, the presence of a rational or elliptic curve immediately implies non-hyperbolicity. Complementary to such algebraic obstructions, another circle of techniques for investigating complex hyperbolicity has been developed, based on jet differentials~\cite{Bloch1926, Siu2004}.

In arbitrary dimension $n$, a general hypersurface $X \subset \mathbb{P}^{n+1}$ of sufficiently high degree is known to be Kobayashi hyperbolic, thanks to the fundamental works of Siu~\cite{Siu2015} and Brotbek~\cite{Brotbek2017} (cf.\ \cite{Deng,Demailly2020,MerkerTa}). B\'erczi--Kirwan~\cite{BercziKirwan2024} improved the degree bound to be polynomial in $n$ by combining their methods with a result of Riedl--Yang~\cite{Riedl-Yang2022} (cf.\ \cite{Cadorel}). 

In dimension $n=2$, the current record bound is $286$, obtained by combining Demailly's \cite[Th.~10.8]{Demailly2020} with Riedl--Yang's work~\cite{Riedl-Yang2022}. If one relaxes ``general'' to ``very general,'' significantly lower bounds were known earlier via different methods. McQuillan~\cite{McQuillan1998, McQuillan1999} established that a very general surface of degree $d \geq 36$ in $\mathbb{P}^3$ is Kobayashi hyperbolic. This bound was subsequently improved to $d \geq 21$ by Demailly--El Goul~\cite{Demailly-ElGoul2000} and to $d \geq 18$ by P\u{a}un~\cite{Paun2008}\footnote{for $d\geq 17$ by Hou--Huynh--Merker--Xie~\cite{Hou-Huynh-Merker-Xie2026} and for $d\geq 16$ by Cui--Hou--Liu--Xie\cite{CuiHouLiuXie2026}.}. These results rely on Clemens' theorem~\cite{Clemens1986} and thus hold only under the very general hypothesis.

\subsection{Foliations and Poincar\'e problem}
A third perspective, initiated in the seminal work of McQuillan~\cite{McQuillan1998}, relates entire curves to foliations. This work marked one of the principal starting points of the modern birational theory of foliations (see \cite{BrunellaSurvey,McQuillan2008}). 

A classical problem in foliation theory, raised by Poincar\'e~\cite{Poincare},
asks whether an algebraic differential equation in two variables admits a
rational first integral. Poincar\'e observed that a decisive step toward such a criterion would be to bound the degree of a possible algebraic solution in terms of the degree of the differential equation itself. This perspective later evolved into the problem of bounding the degree of invariant algebraic curves of a holomorphic foliation purely in terms of its numerical invariants. 

Lins Neto~\cite{LinsNeto2002} showed that no such bound exists in full
generality, constructing families of foliations on $\bP^2$ whose algebraically integrable members admit rational first integrals of unbounded degrees and whose general leaves have unbounded genera. Nevertheless, under additional geometric input, the Poincar\'e problem can often be answered; see \cite{CerveauLinsNeto,Carnicer1994,Pereira2001,GenzmerMol,PereiraSvaldi,SpicerSvaldi,LuTan2024,Vassiliadis}. We will show that some results of Lu--Miyaoka~\cite{LuMiyaoka} solve the Poincar\'e problem for foliations on surfaces of general type that arise in our study of hyperbolicity. 

\subsection{Main results: from ``Very General'' to ``General''}
\subsubsection{Hypersurfaces in $\bP^3$}
Demailly and El Goul~\cite[Remark~7.2]{Demailly-ElGoul2000} raised the challenge of strengthening hyperbolicity from ``very general'' to ``general'' surfaces $X\subset\mathbb{P}^3$ when the degree is at least $21$. We give an affirmative answer:

\begin{mainthm}\label{mainthm:kobayashi-hypersurface}
A general surface $X \subset \mathbb{P}^3$ of degree at least $18$ is Kobayashi hyperbolic.\footnote{The bound $18$ can be replaced by $16$ if we use the preprints \cite{Hou-Huynh-Merker-Xie2026,CuiHouLiuXie2026} instead of \cite{Demailly-ElGoul2000,Paun2008}.} 
\end{mainthm}

This improves the previously known bound for general surfaces from $\deg X\geq 286$ to $\deg X\geq 18$. 
The following is our main result underpinning this improvement:

\begin{mainthm}\label{mainthm:main-algebraic-degeneracy}
Let $p:\cX\to\cP$ be a smooth projective morphism over a smooth complex quasi-projective variety $\cP$, with smooth surface fibers
of general type. Assume that a very general fiber contains no rational
or elliptic curves. Assume that there exist an integer $m\geq 1$,
a line bundle $\aL$ on $\cX$, and a nonzero relative twisted
symmetric differential
\[
\boldsymbol\omega \in 
\coh^0\bigl(\cX, \Sym^m\Omega^1_{\cX_{\cP}/\cP}\otimes \aL\bigr)
\]
with the following property: for every $t \in \cP$ and every entire curve 
$f:\bC \to X_t$, one has \(f^*(\boldsymbol\omega\vert_{X_t})=0\). Then there exists a nonempty Zariski-open subset of $\cP$ such that, for every $t$ in this subset, the surface $X_t$ contains no rational 
or elliptic curves.
\end{mainthm}

The hypothesis of Theorem~\ref{mainthm:main-algebraic-degeneracy} holds for hypersurfaces in $\bP^3$ of $d\geq 18$, as explained in Section~\ref{sec:jet-to-symmetric}. This formulation uses an argument suggested by Yum-Tong Siu and the construction of invariant $2$-jet differentials established for $d\geq 21$ by Demailly--El Goul~\cite{Demailly-ElGoul2000}, for $d\geq 18$ by P\u{a}un~\cite{Paun2008}.\footnote{for $d\geq 17$ by Hou--Huynh--Merker--Xie~\cite{Hou-Huynh-Merker-Xie2026} and for $d\geq 16$ by Cui--Hou--Liu--Xie\cite{CuiHouLiuXie2026}.} 
It follows that:
\begin{mainthm}\label{mainthm:no-rational-curves} 
A general surface $X \subset \mathbb{P}^3$ of degree at least $18$ contains no rational or elliptic curves. 
\end{mainthm}

Theorem~\ref{mainthm:no-rational-curves} strengthens Clemens' classical result from ``very general'' to ``general'' when $\deg X \ge 18$. 
The proof of Theorem~\ref{mainthm:kobayashi-hypersurface} follows from McQuillan's celebrated work~\cite{McQuillan1998} combined with Theorem~\ref{mainthm:main-algebraic-degeneracy}. McQuillan proved that any entire curve $f:\bC\rightarrow X_t$ in a projective surface $X_t$ of general type that is tangent to a multi-foliation (e.g., $f^*(\boldsymbol\omega\vert_{X_t})=0$) must be algebraically degenerate. Therefore its image $f(\bC)$ is contained in some rational or elliptic curve. Once the hypothesis of Theorem~\ref{mainthm:main-algebraic-degeneracy} is satisfied, such curves cannot exist, which yields hyperbolicity. More generally, we have:

\begin{mainthm}\label{mainthm:hypersurface-hyperbolicity}
Under the assumptions of Theorem~\ref{mainthm:main-algebraic-degeneracy}, for a general parameter $t \in \cP$, the surface $X_t$ is Kobayashi hyperbolic.
\end{mainthm}

Hou--Huynh--Merker--Xie~\cite{Hou-Huynh-Merker-Xie2026} and Cui--Hou--Liu--Xie~\cite{CuiHouLiuXie2026} developed a computer-assisted method verifying the hypothesis of Theorem~\ref{mainthm:main-algebraic-degeneracy} for $\deg X\geq 16$. Their approach, based on algorithmic refinements and increased computational power, is expected to further lower the bound to $15$, which appears to be the natural limit of the $2$-jet method. Since the implementation is computationally intensive and beyond the scope of this exposition, we state our main results under the more convenient assumption $\deg X\geq 18$.

\subsubsection{Curve complements}
Our method also applies to the study of Kobayashi hyperbolicity of open surfaces. As in the case of hypersurfaces, the study of hyperbolicity of complements of divisors in the plane has a long history. In the case of reducible boundaries in certain degree ranges, Xi Chen~\cite{Chenlog} proved logarithmic algebraic hyperbolicity for very general unions of plane
curves, while Rousseau~\cite{Rousseau2009} established Kobayashi
hyperbolicity for very general plane curves. More recently, Caporaso--Turchet~\cite{CaporasoTurchet} studied the closely related algebraic
exceptional locus and raised two open questions which form precisely the novelty of Theorem~\ref{mainthm:plane-pair-complement} below. The first asks for results when the boundary has exactly two irreducible components, and was solved in certain degree ranges in~\cite{HouWangXie2026}. The second asks whether the very general condition can be upgraded, which is precisely the purpose of the method developed in this paper.

For integers $1\leq d_1\leq d_2$, let
\begin{equation}\label{eq:plane-pair-parameter-space}
\cP_{d_1,d_2}
:=
\bP\coh^0\bigl(\bP^2,\aO_{\bP^2}(d_1)\bigr)
\times
\bP\coh^0\bigl(\bP^2,\aO_{\bP^2}(d_2)\bigr),
\end{equation}
and let $\cP_{d_1,d_2}^{\mathrm{snc}}$ be the nonempty Zariski-open
subset parametrizing effective divisors $D_a=C_{1,a}+C_{2,a}$ with two smooth irreducible components which intersect transversely.

\begin{mainthm}\label{mainthm:plane-pair-complement}
Assume that
\[
d_1,d_2\geq 3,
\quad\text{or}\quad
d_1=2,\ d_2\geq 5,
\quad\text{or}\quad
d_1=1,\ d_2\geq 8.
\]
There is a nonempty Zariski-open subset
$\cU_{d_1,d_2}\subset\cP_{d_1,d_2}^{\mathrm{snc}}$ such that, for every
$a\in\cU_{d_1,d_2}$ and every rational curve $\Gamma\not\subset D_a$ with normalization $\nu:\bP^1\to\Gamma$, one has \(\#\operatorname{Supp}\nu^*D_a\geq 3\).
\end{mainthm}

Hou--Wang--Xie~\cite{HouWangXie2026} proved hyperbolicity for a very general parameter while Theorem~\ref{mainthm:plane-pair-complement} upgrades their very general condition to a Zariski-open one. As a consequence, we have:

\begin{mainthm}\label{mainthm:plane-complement-hyperbolicity}
Under the degree assumptions of Theorem~\ref{mainthm:plane-pair-complement}, there is a nonempty Zariski-open subset of $\cP_{d_1,d_2}^{\mathrm{snc}}$ such that $\bP^2\setminus D_a$ is Kobayashi hyperbolic and hyperbolically embedded in $\bP^2$ for every $a$ in this subset.
\end{mainthm}

\subsection{Strategy}
One can build a relative twisted symmetric differential form on the universal family of surfaces by using invariant jet-differentials. For general parameters, the eigen-directions of this twisted symmetric differential form a multi-foliation. On the projectivized tangent bundle, the twisted symmetric differential defines a divisor along which the multi-foliation becomes a genuine foliation. By construction, rational and elliptic curves on a surface lift to curves tangent to this induced family of foliations. For these foliations, we solve the Poincaré problem by giving a uniform bound for degrees of invariant curves. Then classical arguments with Hilbert schemes turn such a degree bound into our desired Zariski-open condition. 

The present work is part of a broader effort, in the spirit initiated by McQuillan, to integrate foliation theory more deeply into the study of complex hyperbolicity. 

\subsection{Plan}
In Section~\ref{sec:single-symmetric-differential}, we recall McQuillan's construction of a foliation induced by twisted symmetric differentials. In Section~\ref{sec:family-hypersurfaces}, we carry out the relative version of this construction over the parameter space of surfaces and establish the necessary preparatory results. In Section~\ref{sec:multi-foliation}, we review the construction of twisted symmetric differentials from invariant $2$-jet differentials. The Poincar\'e-type degree bound used in the argument is proved in Section~\ref{sec:poincare}, which may be read independently of the rest of the paper. The proofs of Th.~\ref{mainthm:kobayashi-hypersurface}, Th.~\ref{mainthm:main-algebraic-degeneracy}, Th.~\ref{mainthm:no-rational-curves} and Th.~\ref{mainthm:hypersurface-hyperbolicity} are given in Section~\ref{sec:main-proofs-hypersurfaces} while those of Th.~\ref{mainthm:plane-pair-complement} and Th.~\ref{mainthm:plane-complement-hyperbolicity} are given in Section~\ref{sec:curve-complements}.

\section{From Symmetric Differentials to Foliations}\label{sec:single-symmetric-differential}

Let $X$ be a smooth projective surface. Following McQuillan's construction~\cite{McQuillan1998,McQuillan1999}, we explain how a non-zero twisted symmetric differential form 
\[
\omega\in\coh^0(X,\Sym^m\Omega^1\otimes L)
\]
induces a holomorphic foliation on certain auxiliary surface. We will carry out this construction in family in \S~\ref{sec:family-hypersurfaces}. Though the exposition in \S~\ref{sec:family-hypersurfaces} is essentially self-contained, we explain in detail the case of a single surface in this section for illustration. 

The idea of the construction is as follows. Trivializing \(L\) and working away from the zero locus and the discriminant of the reduced polynomial in the tangent directions, one may factor locally analytically \(\omega=u\prod_{i=1}^r\alpha_i^{a_i}\), where \(u\) is invertible, the holomorphic \(1\)-forms \(\alpha_i\) are nowhere zero and pairwise nonproportional, and \(a_i\geq 1\), \(\sum_i a_i=m\). The equations \(\alpha_i(v)=0\) define local branches of one holomorphic multi-foliation.

\subsection{The projectivized tangent bundle}

Let \(\pi:\mathbb P(T_{X})\to X\) be the bundle of lines in the tangent spaces of $X$. A point of $\mathbb P(T_{X})$ is a pair $(x,\ell)$, where $x\in X$ and $\ell\subset T_{X,x}$ is a $1$-dimensional $\mathbb{C}$-vector subspace. We have the tautological exact sequence  \begin{equation}\label{eq:tautological-exact-sequence-tangent-bundle}
0\longrightarrow \mathcal O_{\mathbb P(T_X)}(-1) \longrightarrow \pi^*T_X \longrightarrow \mathcal Q \longrightarrow 0, 
\end{equation}
where the fiber of $\mathcal O_{\mathbb P(T_X)}(-1)$ at $(x,\ell)$ is the line $\ell$ and $\mathcal Q$ is a line bundle. Taking determinants in \eqref{eq:tautological-exact-sequence-tangent-bundle} we get 
\[ 
\mathcal Q \simeq \mathcal O_{\mathbb P(T_X)}(1)\otimes\pi^*\det(T_X) = \mathcal O_{\mathbb P(T_X)}(1)\otimes\pi^*K_X^{-1}. 
\] 
The pullback of $\omega$ defines a morphism 
\[ 
\Sym^m\mathcal O_{\mathbb P(T_X)}(-1) = \mathcal O_{\mathbb P(T_X)}(-m) \longrightarrow \pi^*L. 
\] 
Equivalently, it defines a section 
\[ 
s_\omega\in \coh^0\!\left( \mathbb P(T_X), \mathcal O_{\mathbb P(T_X)}(m)\otimes\pi^*L \right). 
\] 
We denote its zero scheme by \(D_\omega:=Z(s_\omega)\subset\mathbb P(T_X)\).
Since $\omega\neq 0$, the section $s_\omega$ is not identically zero. As $\mathbb P(T_X)$ is smooth and integral, $D_\omega$ is an effective Cartier divisor.

\subsection{The canonical contact structure}

Composing the differential of $\pi$ with the tautological quotient defines a surjective morphism of vector bundles 
\[ 
\theta:T_{\bP(T_X)} \xrightarrow{\,d\pi\,} \pi^*T_X \longrightarrow \mathcal Q. 
\] 
In other words, \(\theta\in \coh^0\!\left( \mathbb P(T_X), \Omega^1_{\mathbb P(T_X)}\otimes\mathcal Q \right)\). 
At a point $(x,\ell)\in\mathbb P(T_X)$, it is given by 
\[ 
\theta_{(x,\ell)}(\xi) = d\pi(\xi)\bmod\ell \quad \in T_{X,x}/\ell = \mathcal Q_{(x,\ell)}. 
\]
We verify the contact condition locally analytically. Let $(x,y)$ be local analytic coordinates on $X$, and write a tangent direction as 
\[ 
[p:q]=[p\partial_x+q\partial_y]. 
\] 
On the affine chart $q\neq 0$, denote $u=p/q$. The tautological line is generated by \(u\partial_x+\partial_y\), and the class of $\partial_x$ is a local frame of $\mathcal Q$. If some $\xi$ satisfies $d\pi(\xi)=a\partial_x+b\partial_y$, then 
\[ 
d\pi(\xi)\bmod \langle u\partial_x+\partial_y\rangle = (a-ub)[\partial_x]. 
\] 
With respect to the local frame $[\partial_x]$ of $\mathcal Q$, the form $\theta$ satisfies therefore 
\[ 
\theta_{\mathrm{loc}}=dx-u\,dy, \quad
d\theta_{\mathrm{loc}}=-du\wedge dy, \quad
\theta_{\mathrm{loc}}\wedge d\theta_{\mathrm{loc}}=-dx \wedge du \wedge dy.
\] 
Thus $\theta_{\mathrm{loc}}\wedge d\theta_{\mathrm{loc}}$ is nowhere vanishing, and \(\mathcal H:=\ker(\theta)\subset T_{\mathbb P(T_X)}\) is a rank-$2$ contact distribution on $\mathbb P(T_X)$.

\subsection{The characteristic foliation}

Let $Y$ be a reduced irreducible component of $D_\omega$ such that
$\pi(Y)=X$, and let \(\nu:\widetilde Y\to Y\) be its normalization. Denote by
\[
f:\widetilde Y\xrightarrow{\nu}Y
\hookrightarrow\mathbb P(T_X)
\]
the induced morphism, and set \(g:=\pi\circ f:\widetilde Y\to X\). 
Since $g$ is a dominant morphism between irreducible surfaces, it is
generically finite. As the ground field has characteristic zero, $dg$
is an isomorphism at the generic point of $\widetilde Y$.

Pulling back the contact form and composing with the differential of
$f$, we obtain a morphism
\[
\varphi:
T_{\widetilde Y}
\xrightarrow{\,df\,}
f^*T_{\mathbb P(T_X)}
\xrightarrow{\,f^*\theta\,}
f^*\mathcal Q.
\]
Also, $\varphi$ is the composition
\[
T_{\widetilde Y}
\xrightarrow{\,dg\,}
g^*T_X
\longrightarrow
f^*\mathcal Q.
\]
It is surjective at the generic point. Define \(\mathcal F:=\ker(\varphi)\subset T_{\widetilde Y}\). Then $\mathcal F$ has generic rank $1$, and
\[
\mathcal Q_{\mathcal F}:=T_{\widetilde Y}/\mathcal F
\simeq
\operatorname{im}(\varphi)
\subset f^*\mathcal Q
\]
is torsion-free because it is a subsheaf of a line bunlde. Thus $\mathcal F$ is saturated in $T_{\widetilde Y}$.

The tangent sheaf $T_{\widetilde Y}$ is reflexive by \cite[Prop.~1.1]{HartshorneReflexive}. Then \cite[Prop.~1.1]{HartshorneReflexive} yields, locally on $\widetilde Y$, an exact sequence
\[
0\longrightarrow
T_{\widetilde Y}
\longrightarrow
\mathcal V
\longrightarrow
\mathcal G
\longrightarrow 0,
\]
where $\mathcal V$ is locally free and $\mathcal G$ is torsion-free.
The inclusions \(\mathcal F\subset T_{\widetilde Y}\subset\mathcal V\) 
induce an exact sequence
\[
0\longrightarrow
\mathcal Q_{\mathcal F}
\longrightarrow
\mathcal V/\mathcal F
\longrightarrow
\mathcal G
\longrightarrow 0.
\]
Since an extension of torsion-free sheaves on an integral scheme is
torsion-free, $\mathcal V/\mathcal F$ is torsion-free. Applying again
\cite[Prop.~1.1]{HartshorneReflexive} to
\[
0\longrightarrow
\mathcal F
\longrightarrow
\mathcal V
\longrightarrow
\mathcal V/\mathcal F
\longrightarrow 0
\]
we deduce that $\mathcal F$ is reflexive.

Recall that \(T_{\widetilde Y}=\operatorname{Der}_{\mathbb C}
(\mathcal O_{\widetilde Y},\mathcal O_{\widetilde Y})\). Since $\mathcal F\subset T_{\widetilde Y}$, the Lie bracket of two local sections of $\mathcal F$ is a local section of $T_{\widetilde Y}$. Moreover, for local functions $a,b$ and local
sections $\xi,\eta$ of $\mathcal F$, one has
\[
[a\xi,b\eta]
=
ab[\xi,\eta]+a\xi(b)\eta-b\eta(a)\xi.
\]
As the last two terms lie in $\mathcal F$, the pairing \((\xi,\eta)\mapsto[\xi,\eta]\bmod\mathcal F\) is alternating and $\mathcal O_{\widetilde Y}$-bilinear. It therefore
induces an $\mathcal O_{\widetilde Y}$-linear morphism
\begin{equation}\label{eq:lie-bracket-morphism}
{\bigwedge}^2\mathcal F
\longrightarrow
T_{\widetilde Y}/\mathcal F.
\end{equation}
Since $\mathcal F$ has generic rank $1$, the sheaf
$\bigwedge^2\mathcal F$ is torsion. Recall that 
$T_{\widetilde Y}/\mathcal F$ is torsion-free. Hence the morphism \eqref{eq:lie-bracket-morphism}
vanishes. Thus $\mathcal F$ is closed under the Lie bracket and defines
a saturated rank-$1$ foliation on $\widetilde Y$.

Finally, let $U\subset X$ be an analytic open subset on
which $\omega=u\prod_{i=1}^r\alpha_i^{a_i}$ with the $\alpha_i$
nowhere zero and pairwise nonproportional. As a divisor,
$D_\omega|_{\pi^{-1}(U)}=\sum_{i=1}^r a_i\sigma_i(U)$, and its
support is the union of the graphs 
\[
\sigma_i:U\to\bP(T_X),
\quad
\sigma_i(x)=\ker(\alpha_i)_x.
\]
If we identify the graph of $\sigma_i$ with $U$, then the
restriction of $\varphi$ is the natural quotient morphism \(T_U\to T_U/\ker(\alpha_i)\). Therefore the restriction of $\mathcal F$ to this local chart is $\ker(\alpha_i)$. In summary, the foliations constructed on the normalizations of the components of $D_\omega$ recover the local branches of the $m$-multi-foliation defined by $\omega$.

\section{Families of Surfaces With Multi-foliations}\label{sec:family-hypersurfaces}
In this section, let $\cP$ be a smooth complex quasi-projective variety and let $p:\cX\to\cP$ be a smooth projective morphism with smooth 
surface fibers.  Let $\aH$ be a $p$-ample line bundle.  For $t\in\cP$, we
denote the fiber by $X_t$ and the restriction of $\aH$ by $H_t$.

Let $m\geq 1$, let $\aL$ be a line bundle on $\cX$. Consider a nonzero twisted relative symmetric differential on the total space
\[
\boldsymbol\omega\in
\coh^0\bigl(\cX,\Sym^m\Omega^1_{\cX/\cP}\otimes\aL\bigr)
\]
whose restriction on $X_t$ is a twisted symmetric differential form 
\[
\omega_t:=\boldsymbol\omega\vert_{X_t}\in \coh^0\bigl(X_t,\Sym^m\Omega^1\otimes L_t\bigr).
\]
We assume that $\omega_t$ is nonzero for every $t\in\cP$. This assumption can be always achieved up to replacing $\cP$ by a nonempty Zariski-open subset.

\subsection{The spectral divisor}\label{subsec:spectral-divisor}

Let
\[
\Pi:\bP\bigl(T_{\cX/\cP}\bigr)\longrightarrow \cX
\]
be the relative projectivized tangent bundle. Denote by $\aO_{\bP}(1)$ the hyperplane line bundle (i.e., dual of the tautological line bundle $\aO_{\bP}(-1)$).

Fiberwise over a point $x\in X_t$, the symmetric differential $\omega_t(x)$ is a homogeneous polynomial of degree $m$ on $T_{X_t,x}$, equivalently an element of $\aO_{\bP(T_{X_t,x})}(m)$.
Globally, $\boldsymbol\omega$ induces a canonical section
\[
\boldsymbol s_{\boldsymbol\omega}\in \coh^0\Bigl(\bP\bigl(T_{\cX/\cP}\bigr),\,
\aO_{\bP}(m)\otimes \Pi^*\aL\Bigr).
\]
Define the \emph{spectral divisor}:
\[
\cD := \Div(\boldsymbol s_{\boldsymbol\omega}) \subset \bP\bigl(T_{\cX/\cP}\bigr).
\]
It is an effective Cartier divisor. It is flat over a non-empty open subset of $\cP$ because flatness is a generic property~\cite[Th.~6.9.1]{EGAIV}. Its restriction to a fiber is called the \emph{spectral divisor defined by $\omega_t$}:
\[
D_{\omega_t}\subset\bP(T_{X_t}).
\]

\subsection{Base locus}\label{sec:horizontal-spectral-divisor}
Let $\cZ\subset \cX$ be the zero locus of $\boldsymbol{\omega}$. For $t\in\cP$, the intersection
\[
Z(\omega_t):=\cZ\cap X_t \subset X_t
\]
is the base locus of $\omega_t$. By construction, $x\in Z(\omega_t)$ if and only if $\omega_t(x)=0$ in $\Sym^m\Omega^1_{X_t,x}\otimes L_{t,x}$. 

\begin{lem}\label{lem:base-locus-semicontinuity}
For every integer $k\geq 0$, the locus
\[
\Sigma_{\geq k}:=\{t\in\cP\mid \dim Z(\omega_t)\ge k\}
\]
is a Zariski closed subset of $\cP$. In particular,
\[
\Sigma:=\Sigma_{\ge 1}=\{t\in\cP\mid Z(\omega_t)\ \text{contains a curve}\}
\]
is Zariski closed.
\end{lem}
\begin{proof}
Note that $\cZ$ is closed in $\cX$ and the morphism $\cZ\to\cP$ is projective. Then it suffices to apply \cite[Cor.~13.1.5]{EGAIV}.
\end{proof}

\begin{lem}\label{lem:uniform-base-locus-degree}
Up to replacing $\cP$ by a nonempty Zariski-open subset, there is an integer $\tau_Z\geq 0$ such that \(H_t\cdot C\leq\tau_Z\) for every $t\in\cP$ and every curve $C\subset Z(\omega_t)$.
\end{lem}

\begin{proof}
Let $\cZ\subset\cX$ be the zero scheme of $\boldsymbol\omega$. Since
$\omega_t$ is nonzero for every $t$, every fiber of
$\cZ\to\cP$ has dimension at most one. We remove the images of the irreducible components of $\cZ$ that do not dominate $\cP$. For every remaining component whose generic fiber is zero-dimensional, Lemma~\ref{lem:base-locus-semicontinuity} allows us to remove the locus where its fiber has positive dimension. 
Let $\cZ^{(1)}$ be the reduced union of the remaining components whose
generic fibers have dimension one. By \cite[Th.~6.9.1]{EGAIV} we can assume, up to shrinking, that $\cZ^{(1)}\to\cP$ is flat. Then the
$H_t$-degree of $\cZ_t^{(1)}$ is our desired constant.
\end{proof}

\begin{lem}\label{lem:etale-base-change-irreducible-components}
Up to replacing $\cP$ with a non-empty open subset, there is a finite surjective étale morphism \(u:\cP^\circ\to\cP\) such that every irreducible component of the geometric generic fiber (over the closure of the function field) of the base change $\cP^\circ\times_{\cP}\cD$ is defined over $\bC(\cP^\circ)$.
\end{lem}

\begin{proof}
Let $\eta$ be the generic point of $\cP$, let
\[
K:=\bC(\cP),
\qquad
\cQ:=\bP\bigl(T_{\cX/\cP}\bigr).
\]
We write the divisor induced by $\cD$ on the geometric generic fiber over $\overline{K}$ as
\[
\cD_\eta
=
\sum_{j=1}^{s} a_j E_{j,\eta}
+
\sum_{\ell=1}^{r} b_\ell F_{\ell,\eta},
\]
where the $E_{j,\eta}$ are the irreducible components whose
images under \(\Pi_\eta:\cQ_\eta\to X_\eta\) are dense in $X_\eta$, whereas the $F_{\ell,\eta}$ do not dominate $X_\eta$.

Consider a finite extension $K'/K$ over which all reduced supports of these components are defined. Since $\operatorname{char}K=0$, we may, after enlarging $K'$ if necessary, assume that $K'/K$ is finite Galois. Up to replacing $\cP$ with a nonempty affine open subset, we can and will assume that \(\cP=\Spec A\) where $A$ is an integral finitely generated $\bC$-algebra with fraction field $K$. By the primitive element theorem, there exists
$\alpha\in K'$ such that \(K'=K(\alpha)\). Up to replacing $A$ by a localization $A_f$, for some nonzero $f\in A$, we can and will assume that $\alpha$ is integral over $A$. Thus \(B:=A[\alpha]\subset K'\) is a finite integral $A$-algebra. It is an integral domain, and its fraction field is $K'$. The induced morphism
\[
u_0:\widetilde{\cP}:=\Spec B\to\cP
\]
is finite and surjective.

Let $\widetilde{\eta}$ be the generic point of $\widetilde{\cP}$. The
fiber of $u_0$ over $\eta$  consists of the single point $\widetilde{\eta}$. At this
point, $u_0$ corresponds to the morphism
\(\Spec K'\to\Spec K\). By \cite[Th.~17.6.1]{EGAIV} the map $u_0$ is \'etale at
$\widetilde{\eta}$. By the definition of a morphism being \'etale at a point
\cite[Def.~17.3.7]{EGAIV}, there exists an open neighborhood \(W\subset\widetilde{\cP}\) of $\widetilde{\eta}$ such that the restriction \(u_0|_W:W\to\cP\) is \'etale. Then the complement \(R:=\widetilde{\cP}\setminus W\) has closed image in $\cP$. Hence \(\cP_1:=\cP\setminus u_0(R)\) is a nonempty open subset of $\cP$. 

Define \(\cP^\circ:=u_0^{-1}(\cP_1)\). Then $\cP^\circ\subset W$, and the restriction \(u:\cP^\circ\to\cP_1\) is finite, étale and surjective. 
Replacing $\cP$ by $\cP_1$, we have therefore obtained a finite
surjective \'etale morphism \(u:\cP^\circ\to\cP\) whose extension of function fields is $K'/K$. In particular, after
this base change, every irreducible component of the geometric
generic fiber of $\cD$ is defined over $\bC(\cP^\circ)$.
\end{proof}

Let \(u:\cP^\circ\to\cP\) be the finite surjective étale morphism of Lemma~\ref{lem:etale-base-change-irreducible-components}. We define
\begin{equation}\label{eq:definition-Q}
\cX^\circ:=\cX\times_{\cP}\cP^\circ,
\quad
\cQ^\circ:=\bP\bigl(T_{\cX^\circ/\cP^\circ}\bigr),
\end{equation}
and denote by
\begin{equation}\label{eq:definition-Pi}
\Pi^\circ:\cQ^\circ\longrightarrow\cX^\circ
\end{equation}
the projection. Via the canonical isomorphism \(\cQ^\circ\simeq\bP\bigl(T_{\cX/\cP}\bigr)\times_{\cP}\cP^\circ\), we set 
\[
\cD^\circ:=\cD\times_{\cP}\cP^\circ
\]
to be the induced spectral divisor.

Let $\eta^\circ$ be the generic point of $\cP^\circ$. We can write the decomposition of the divisor on the generic fiber as \(\cD^\circ_{\eta^\circ}=\sum_{i=1}^{N}a_iE_{i,\eta^\circ}\) 
where the components are numbered so that \(E_{1,\eta^\circ},\ldots,E_{s,\eta^\circ}\) are those which dominate $X^\circ_{\eta^\circ}$ under
$\Pi^\circ_{\eta^\circ}$. For $1\leq i\leq N$, let \(\cY_i^\circ :=\overline{E_{i,\eta^\circ}}\) be the scheme-theoretic closure of $E_{i,\eta^\circ}$ in $\cQ^\circ$, and define
\[
\cD^{\circ,\mathrm{hor}}
:=
\sum_{j=1}^{s}a_j\cY_j^\circ.
\]
For $t\in\cP^\circ$, we denote by \(X_t^\circ,\cQ_t^\circ,D_t^\circ,Y_{i,t}^\circ\) the fibers of $\cX^\circ$, $\cQ^\circ$, $\cD^\circ$, and $\cY_i^\circ$, respectively.

\begin{lem}\label{lem:horizontal-spectral-divisor-family}
Up to replacing $\cP$ by a non-empty open subset and $\cP^\circ$ by its inverse image under $u$, the following assertions hold.
\begin{enumerate}
\item[(i)]
The cycle $\cD^{\circ,\mathrm{hor}}$ is an effective Cartier divisor
on $\cQ^\circ$ and is flat over $\cP^\circ$. For every
$1\leq j\leq s$, the morphism \(\cY_j^\circ\to\cP^\circ\) is flat with integral fibers.
\item[(ii)]
For every $1\leq j\leq s$, the restriction \(\Pi^\circ|_{\cY_j^\circ}:
\cY_j^\circ\to\cX^\circ\) is surjective.
\item[(iii)]
For every $t\in\cP^\circ$, the divisors \(Y_{1,t}^\circ,\cdots,Y_{s,t}^\circ\) are pairwise distinct and are the irreducible components of
$D_t^\circ$ which dominate $X_t^\circ$ under $\Pi_t^\circ$.
\item[(iv)]
For every $t\in\cP^\circ$, the morphisms \(\Pi_t^\circ|_{Y_{j,t}^\circ}:
Y_{j,t}^\circ\to X_t^\circ\) are dominant, generically finite, and finite over \(X_t^\circ\setminus Z(\omega_{u(t)})\). They also satisfy
\[
\sum_{j=1}^{s}
a_j\deg\bigl(\Pi_t^\circ|_{Y_{j,t}^\circ}\bigr)=m
\]
where $m$ is the symmetric degree of $\boldsymbol\omega$. In particular, \(\deg\bigl(\Pi_t^\circ|_{Y_{j,t}^\circ}\bigr)
\leq m\) for every $j$.
\end{enumerate}
\end{lem}

\begin{proof}
In this proof, ``up to shrinking'' will mean replacing $\cP$ with a
non-empty open subset and $\cP^\circ$ with the corresponding preimage. 
Note that $\cQ^\circ$ is a
projective bundle over $\cX^\circ$, thus smooth over
$\bC$. 

By Lemma~\ref{lem:etale-base-change-irreducible-components}, every
irreducible component of
\(
\cD^\circ_{\eta^\circ}
\)
is geometrically irreducible. 
Consequently, every $E_{i,\eta^\circ}$ is geometrically integral.

The union of irreducible components of $\cD^\circ$ which do not dominate $\cP^\circ$ has proper closed image in $\cP^\circ$. Therefore up to shrinking wa can assume that every irreducible component of $\cD^\circ$ dominates $\cP^\circ$. There is a bijection between
the irreducible components of $\cD^\circ$ and those of
$\cD^\circ_{\eta^\circ}$. 
In particular \(\cY_1^\circ,\ldots,\cY_N^\circ\) are precisely the irreducible components of $\cD^\circ$. Each
$\cY_i^\circ$ is a prime Weil divisor on the smooth variety $\cQ^\circ$.
By \cite[Prop.~II.6.11]{HartshorneGTM52} it is an effective Cartier divisor. The multiplicity of $\cD^\circ$ along
$\cY_i^\circ$ equals the multiplicity of
$\cD^\circ_{\eta^\circ}$ along $E_{i,\eta^\circ}$. In particular, we have
\[
\cD^\circ
=
\sum_{i=1}^{N}a_i\cY_i^\circ
\]
as Cartier divisors on $\cQ^\circ$. Define
\[
\cD^{\circ,\mathrm{ver}}
:=
\sum_{i=s+1}^{N}a_i\cY_i^\circ,
\]
so \(\cD^\circ=\cD^{\circ,\mathrm{hor}}+\cD^{\circ,\mathrm{ver}}\).
As flatness is a generic property~\cite[Th.~6.9.1]{EGAIV}, up to shrinking we can and will assume that the morphisms \(\cD^{\circ,\mathrm{hor}}\to\cP^\circ, \cD^{\circ,\mathrm{ver}}\to\cP^\circ, \cY_i^\circ\to\cP^\circ,1\leq i\leq N\) are flat.

For $1\leq j\leq s$, by \cite[Th.~12.2.4(viii)]{EGAIV}, up to shrinking, every fiber $Y_{j,t}^\circ$ of the morphism \(\cY_j^\circ\to\cP^\circ\) is geometrically integral. This proves (i).

For $j\neq k$, the generic fiber of \(\cY_j^\circ\cap\cY_k^\circ
\to\cP^\circ\) has dimension at most one because $E_{j,\eta^\circ}$ and $E_{k,\eta^\circ}$ are distinct prime divisors in the threefold $\cQ^\circ_{\eta^\circ}$. This morphism is projective. Hence, by the upper semicontinuity of fiber dimension
\cite[Cor.~13.1.5]{EGAIV}, up to shrinking, \(\dim\bigl(Y_{j,t}^\circ\cap Y_{k,t}^\circ\bigr)\leq 1\) for every $t\in\cP^\circ$. In other words the divisors \(Y_{1,t}^\circ,\ldots,Y_{s,t}^\circ\) are pairwise distinct.

Consider
\[
\cB^\circ
:=
\Pi^\circ\bigl(
\operatorname{Supp}(\cD^{\circ,\mathrm{ver}})
\bigr)
\subset\cX^\circ.
\]
Then $\cB^\circ$ is
closed because $\Pi^\circ$ is projective. Its generic fiber is the union of the images \(\Pi^\circ_{\eta^\circ}(E_{i,\eta^\circ}), s<i\leq N.\)
By the definition of $s$, each of these images is a proper closed subset of the surface $X^\circ_{\eta^\circ}$. Thus \(\dim\cB^\circ_{\eta^\circ}\leq 1\). By applying again \cite[Cor.~13.1.5]{EGAIV}, we deduce that, up to shrinking, \(\dim \cB_t^\circ\leq 1\) for every $t\in\cP^\circ$. Hence no irreducible component of
$\cD^{\circ,\mathrm{ver}}_t$ dominates $X_t^\circ$.

For $1\leq j\leq s$, the image of \(\Pi^\circ|_{\cY_j^\circ}:
\cY_j^\circ\to\cX^\circ\) is closed. Its intersection with the generic fiber contains \(\Pi^\circ_{\eta^\circ}(E_{j,\eta^\circ})=X^\circ_{\eta^\circ}\). The generic fiber $X^\circ_{\eta^\circ}$ is dense in $\cX^\circ$ because $\cX^\circ\to\cP^\circ$ is smooth while $\cP^\circ$ is irreducible. The surjectivity of Item (ii) follows.

Since $\Pi^\circ|_{\cY_j^\circ}$ is a morphism over $\cP^\circ$, its
surjectivity implies that \(\Pi_t^\circ|_{Y_{j,t}^\circ}:
Y_{j,t}^\circ\to X_t^\circ\) is surjective for every $t$. As its source and target are both surfaces, it is generically finite.

The flatness of $\cD^{\circ,\mathrm{hor}}$ and
$\cD^{\circ,\mathrm{ver}}$ over $\cP^\circ$ allows us to restrict Cartier divisors to every fiber:
\[
D_t^\circ = \sum_{j=1}^{s}a_jY_{j,t}^\circ
+\cD^{\circ,\mathrm{ver}}_t.
\]
The divisors $Y_{j,t}^\circ$ are integral, pairwise distinct, and
dominate $X_t^\circ$, whereas every component of
$\cD^{\circ,\mathrm{ver}}_t$ has image contained in
$\cB_t^\circ$ and therefore does not dominate $X_t^\circ$. This proves
(iii).

Fix $t\in\cP^\circ$. There is a non-empty open subset of $X_t^\circ$ over which all the generically finite morphisms \(\Pi_t^\circ|_{Y_{j,t}^\circ}\) are finite and flat. Let \(x\) be a point in this open subset such that \(x\notin \cB_t^\circ\) and \(x\notin Z(\omega_{u(t)})\). 
Note that the fiber \(L_x:=(\Pi_t^\circ)^{-1}(x)\) 
is isomorphic to $\bP^1$. Since $x\notin Z(\omega_{u(t)})$, $L_x$ is not contained in $D_t^\circ$. Hence \(\deg\bigl(D_t^\circ|_{L_x}\bigr)=m\). Moreover, \(L_x\) does not meet
\(\cD^{\circ,\mathrm{ver}}_t\) because \(x\notin\cB_t^\circ\). For every \(j\), we have \(Y_{j,t}^\circ\cap L_x=\bigl(\Pi_t^\circ|_{Y_{j,t}^\circ}\bigr)^{-1}(x)\), and  therefore \(\length\bigl(Y_{j,t}^\circ\cap L_x\bigr)=\deg\bigl(\Pi_t^\circ|_{Y_{j,t}^\circ}\bigr)\). We then deduce \(m=\sum_{j=1}^{s}
a_j\deg\bigl(\Pi_t^\circ\vert_{Y_{j,t}^\circ}\bigr)\). Item (iv) follows.
\end{proof}

\begin{rem}
In the sequel of this section, we will always work after the shrinking and base change procedure of Lemma~\ref{lem:horizontal-spectral-divisor-family}.
\end{rem}

\subsection{Characteristic foliations}

Fix \(1\leq j\leq s\), and for ease of notation we will write
\[
\cY:=\cY_j^\circ,
\qquad
Y_t:=Y_{j,t}^\circ.
\]
The abusive omitting of the index $j$ is harmless for later need because there are only finitely many of them. 
Let \(\nu:\widetilde{\cY}\to \cY\) be the normalization, and let \(\rho:\cS\to\widetilde{\cY}\) be a resolution of singularities obtained by a sequence of blow-ups along smooth centers. In particular, \(\rho\) is a projective birational morphism and \(\cS\) is smooth. Recall that $\cQ^\circ,\Pi^\circ$ were defined in \eqref{eq:definition-Q}, \eqref{eq:definition-Pi}. We will denote
\[
h:\cS\xrightarrow{\rho}\widetilde{\cY}
\xrightarrow{\nu}\cY
\hookrightarrow\cQ^\circ,
\quad
r:=\Pi^\circ\circ h:\cS\longrightarrow\cX^\circ,
\]
and denote by \(q:\cS\to\cP^\circ\) the structural morphism.

\begin{lem}\label{lem:relative-resolution-family}
Up to replacing $\cP$ by a non-empty Zariski-open subset and
$\cP^\circ$ by its inverse image under $u$, the following hold. 
\begin{enumerate}
\item[(i)]
The morphism \(q:\cS\to\cP^\circ\) is smooth with integral fibers.
\item[(ii)]
For every \(t\in\cP^\circ\), the induced birational morphism \(h_t:S_t\to Y_t\) is a resolution of singularities.
\end{enumerate}
\end{lem}

\begin{proof}
Let \(\eta^\circ\) be the generic point of \(\cP^\circ\). Since
\(h:\cS\to\cY\) is birational, the generic fiber \(S_{\eta^\circ}\) is geometrically integral. 
The varieties \(\cS\) and \(\cP^\circ\) are smooth over \(\bC\).
By \cite[Cor.~III.10.7]{HartshorneGTM52} we obtain a non-empty open subset of \(\cP^\circ\) over which \(q\) is smooth. By
\cite[Th.~12.2.4(viii)]{EGAIV}, this open subset may be chosen so that
all fibers of \(q\) are geometrically integral.

It remains to prove the fiberwise birationality of \(h\). Let \(V\subset\cY\) be a dense open subset such that \(h^{-1}(V)\to V\) is an isomorphism, and consider \(B:=\cY\setminus V\). Since \(V\) contains the generic point of \(\cY\), its generic fiber \(V_{\eta^\circ}\) is a dense open subset of the integral surface
\(Y_{\eta^\circ}\). Thus \(\dim B_{\eta^\circ}\leq 1\). The morphism \(B\to\cP^\circ\) is projective. By upper semicontinuity of
fiber dimension~\cite[Cor.~13.1.5]{EGAIV} we obtain a non-empty open subset over which \(\dim B_t\leq 1\). For every \(t\) in this open subset, \(V_t\) is dense in \(Y_t\), and \(h_t\) is an isomorphism over \(V_t\). Then \(h_t\) is projective, proper and birational while \(S_t\)
is smooth by Item~(i). 

Let \(U^\circ\subset\cP^\circ\) be the intersection of the open
subsets obtained above. Since \(u\) is finite, the image \(u(\cP^\circ\setminus U^\circ)\) is closed in \(\cP\). Its complement \(U\) is a non-empty Zariski-open subset of \(\cP\), and \(u^{-1}(U)\subset U^\circ\). Replacing \(\cP\) by \(U\) and \(\cP^\circ\) by \(u^{-1}(U)\) we prove the lemma.
\end{proof}

In the sequel, we replace \(\cP\) and \(\cP^\circ\) by the open
subsets furnished by Lemma~\ref{lem:relative-resolution-family}. We now carry out the construction of characteristic foliation in \S\ref{sec:single-symmetric-differential} in family.

Recall the morphism from \eqref{eq:definition-Pi}:
\[
\Pi^\circ:\cQ^\circ
=
\bP\bigl(T_{\cX^\circ/\cP^\circ}\bigr)
\longrightarrow\cX^\circ.
\]
Consider the relative tautological exact sequence
\[
0\longrightarrow
\aO_{\cQ^\circ}(-1)
\longrightarrow
(\Pi^\circ)^*T_{\cX^\circ/\cP^\circ}
\longrightarrow
\mathcal Q_{\mathrm{rel}}
\longrightarrow 0.
\]
Its quotient line bundle is
\[
\mathcal Q_{\mathrm{rel}}
\simeq
\aO_{\cQ^\circ}(1)\otimes
(\Pi^\circ)^*K_{\cX^\circ/\cP^\circ}^{-1}.
\]
Composing the relative differential of \(\Pi^\circ\) with the
tautological quotient we obtain the relative contact morphism
\[
\boldsymbol\theta:
T_{\cQ^\circ/\cP^\circ}
\longrightarrow
\mathcal Q_{\mathrm{rel}}.
\]
The relative differential of \(h\) induces a morphism
\[
\boldsymbol\varphi:
T_{\cS/\cP^\circ}
\xrightarrow{\,dh\,}
h^*T_{\cQ^\circ/\cP^\circ}
\xrightarrow{\,h^*\boldsymbol\theta\,}
h^*\mathcal Q_{\mathrm{rel}}
\]
which is equivalently the composition
\[
T_{\cS/\cP^\circ}
\xrightarrow{\,dr\,}
r^*T_{\cX^\circ/\cP^\circ}
\longrightarrow
h^*\mathcal Q_{\mathrm{rel}},
\]
where the second arrow is induced by the tautological quotient.

The morphism \(r_{\eta^\circ}:S_{\eta^\circ}\to X_{\eta^\circ}^\circ\) is dominant and generically finite. Its differential is an isomorphism at the generic point of \(S_{\eta^\circ}\). Hence
\(\boldsymbol\varphi\) has generic rank \(1\). 
Define
\[
T_{\cF}:=\ker(\boldsymbol\varphi)
\subset T_{\cS/\cP^\circ},
\qquad
\mathcal I:=\operatorname{im}(\boldsymbol\varphi)
\subset h^*\mathcal Q_{\mathrm{rel}}.
\]
Then \(T_{\cS/\cP^\circ}/T_{\cF}\simeq\mathcal I\). Since \(\mathcal I\) is a subsheaf of a line bundle on the integral variety \(\cS\), it is torsion-free. Thus \(T_{\cF}\) is a saturated
subsheaf of \(T_{\cS/\cP^\circ}\) of generic rank \(1\).

\begin{lem}\label{lem:characteristic-foliation-base-change}
Up to replacing $\cP$ by a non-empty Zariski-open subset and $\cP^\circ$ by its inverse image under $u$, for every
\(t\in\cP^\circ\) we have
\[
T_{\cF}|_{S_t}
=
\ker\bigl(
\boldsymbol\varphi_t:
T_{S_t}
\longrightarrow
h_t^*\bigl(\mathcal Q_{\mathrm{rel}}\vert_{\cQ_t^\circ}\bigr)
\bigr),
\]
and \(T_{\cF}\vert_{S_t}\) is a saturated rank-\(1\) subsheaf of \(T_{S_t}\).
\end{lem}

\begin{proof}
Denote \(\mathcal C:= h^*\mathcal Q_{\mathrm{rel}}/\mathcal I\). As flatness is generic~\cite[Th.~6.9.1]{EGAIV}, there exists a non-empty Zariski-open subset
\(U^\circ\subset\cP^\circ\) such that \(\mathcal I|_{q^{-1}(U^\circ)}\) and \(\mathcal C|_{q^{-1}(U^\circ)}\) are flat over \(U^\circ\). Since \(u\) is finite, the subset \(U:=\cP\setminus
u\bigl(\cP^\circ\setminus U^\circ\bigr)\) is a non-empty Zariski-open subset of  \(\cP\), and \(u^{-1}(U)\subset U^\circ\). 
We replace \(\cP\) by \(U\) and \(\cP^\circ\) by \(u^{-1}(U)\).

For every \(t\in\cP^\circ\), restriction to the fiber \(S_t\)
preserves the exactness of
\[
0\to
T_{\cF}
\to
T_{\cS/\cP^\circ}
\to
\mathcal I
\to 0
,\quad
0\to
\mathcal I
\to
h^*\mathcal Q_{\mathrm{rel}}
\to
\mathcal C
\to 0.
\]
Since \(q\) is smooth, \(T_{\cS/\cP^\circ}\vert_{S_t}\simeq T_{S_t}\). 
It follows that
\[
T_{\cF}\vert_{S_t}
=
\ker(\boldsymbol\varphi_t)
, \quad
T_{S_t}/T_{\cF}|_{S_t}
\simeq
\mathcal I|_{S_t}
\hookrightarrow
h_t^*\bigl(\mathcal Q_{\mathrm{rel}}|_{\cQ_t^\circ}\bigr).
\]
By Lemma~\ref{lem:relative-resolution-family} \(h_t\) is birational, while \(\Pi_t^\circ|_{Y_t}:Y_t\to X_t^\circ\) is dominant and generically finite by
Lemma~\ref{lem:horizontal-spectral-divisor-family}. Hence \(r_t:S_t\to X_t^\circ\) is dominant and generically finite. Its differential is therefore an
isomorphism at the generic point of \(S_t\), and
\(\boldsymbol\varphi_t\) has generic rank \(1\). Consequently,
\(T_{\cF}|_{S_t}\) has rank \(1\). Its quotient in \(T_{S_t}\) is
torsion-free, being a subsheaf of a line bundle. Hence \(T_{\cF}|_{S_t}\) is saturated.
\end{proof}

In the sequel we replace \(\cP\) and \(\cP^\circ\) by the open subsets furnished by
Lemma~\ref{lem:characteristic-foliation-base-change}. 

As in \S~\ref{sec:single-symmetric-differential} the sheaf \(T_{\cF}\) is closed under the relative Lie bracket.
Indeed, the bracket induces an \(\cO_{\cS}\)-linear morphism
\begin{equation}\label{eq:bracket-morphism-family}
\bigwedge\nolimits^2T_{\cF}
\longrightarrow
T_{\cS/\cP^\circ}/T_{\cF}.
\end{equation}
The left-hand side of \eqref{eq:bracket-morphism-family} is a torsion sheaf because \(T_{\cF}\) has generic rank \(1\), whereas the right-hand side is torsion-free. Hence \eqref{eq:bracket-morphism-family} vanishes.

For \(t\in\cP^\circ\), we denote \(T_{\aF_t}:=T_{\cF}\vert_{S_t}\subset T_{S_t}\). The same argument shows that \(T_{\aF_t}\) is closed under the Lie bracket. Thus it defines a saturated rank-\(1\) foliation \(\aF_t\) on \(S_t\).
On the dense open subset on which \(h_t\) is locally an isomorphism, one has
\[
T_{\aF_t}
=
\ker\!\left(
T_{Y_t}
\xrightarrow{\,
\boldsymbol\theta_t|_{T_{Y_t}}
\,}
\mathcal Q_{\mathrm{rel}}|_{Y_t}
\right).
\]
In summary \(\aF_t\) is the characteristic foliation of
\(Y_t\subset\bP(T_{X_t^\circ})\) pulled back to the resolution \(S_t\).

\subsection{Invariant curves}\label{subsec:invariant-curves}

We restore the index \(j\) in the notation of the preceding
subsection. Since there are only finitely many indices
\(1\leq j\leq s\), we apply
Lemmas~\ref{lem:relative-resolution-family} and
\ref{lem:characteristic-foliation-base-change} simultaneously to all
the components \(\cY_j^\circ\). We write
\[
h_j:\cS_j\longrightarrow \cY_j^\circ\subset\cQ^\circ
\]
for the resulting morphism and, for \(t'\in\cP^\circ\), we denote by
\(\aF_{j,t'}\) the induced characteristic foliation on \(S_{j,t'}\). 

Recall that $\boldsymbol\varphi_j:T_{\cS_j/\cP^\circ}
\to h_j^*\mathcal Q_{\mathrm{rel}}$ is the pulled-back relative
contact morphism. Denote by $\mathcal W_j:=Z(\boldsymbol\varphi_j)$ its zero scheme, viewed as a section of $\Omega^1_{\cS_j/\cP^\circ}\otimes h_j^*\mathcal Q_{\mathrm{rel}}$, and let $W_{j,t'}:=\mathcal W_j\cap S_{j,t'}$.
Note that the morphism $\boldsymbol\varphi_{j,t'}$ has generic rank one (see the proof of Lemma~\ref{lem:characteristic-foliation-base-change}). Therefore we can assume that $\dim W_{j,t'}\leq 1$ for every $t'\in\cP^\circ$.

Let \(t\in\cP\), and let \(C\subset X_t\) be an irreducible curve.
Denote by
\[
\eta_C:\overline C\longrightarrow C\hookrightarrow X_t
\]
its normalization. At a general point \(p\in\overline C\), the
differential \(d\eta_{C,p}\) is non-zero and determines the tangent
line
\[
\ell_p:=
\operatorname{im}\bigl(
d\eta_{C,p}:T_{\overline C,p}\longrightarrow
T_{X_t,\eta_C(p)}
\bigr).
\]
This defines a rational map $\overline C\dashrightarrow\bP(T_{X_t})$, which extends uniquely to a morphism
\[
\gamma_C:\overline C\to\bP(T_{X_t}),
\qquad
p\mapsto \ell_p.
\]
We denote its image by \(\widehat C:=\gamma_C(\overline C)\subset\bP(T_{X_t})\). The equality \(\Pi_t\circ\gamma_C=\eta_C\) holds on a dense open subset of \(\overline C\), and hence everywhere. In particular, we have \(\Pi_t(\widehat C)=C\).

Suppose that \(C\not\subset Z(\omega_t)\). At a point \(x\in X_t\), \(\omega_t(x)\in
\Sym^m T_{X_t,x}^*\otimes L_{t,x}\) defines an \(L_{t,x}\)-valued homogeneous polynomial of degree \(m\)
on \(T_{X_t,x}\). We say that \(C\) is
\emph{invariant by \(\omega_t\)} if, at a general smooth point
\(x\in C\), the restriction of \(\omega_t(x)\) to the tangent line
\(T_{C,x}\) vanishes. 
We can equivalently describe the invariant condition as follows. Pulling back \(\omega_t\) by $\eta_C$ and applying the morphism induced by \(d\eta_C^*:\eta_C^*\Omega^1_{X_t}\to\Omega^1_{\overline C}\), we get a section
\[
\omega_{t,C}
\in
\coh^0\left(
\overline C,\,
(\Omega^1_{\overline C})^{\otimes m}\otimes\eta_C^*L_t
\right).
\]
Then \(C\) is invariant by \(\omega_t\) if and only if \(\omega_{t,C}=0\). 

\begin{lem}\label{lem:invariant-lift}
Let \(t\in\cP\), and let \(C\subset X_t\) be an irreducible curve such
that \(C\not\subset Z(\omega_t)\). Then the following conditions are equivalent:
\begin{enumerate}
\item[(i)]
The curve \(C\) is invariant by \(\omega_t\).
\item[(ii)]
Its tangent lift satisfies \(\widehat C\subset D_{\omega_t}\).
\end{enumerate}
Suppose that these conditions hold. For every
\(t'\in\cP^\circ\) such that \(u(t')=t\), there exists an index \(1\leq j\leq s\) such that \(\widehat C\subset Y_{j,t'}^\circ\). Moreover, if \(C^S\) is an irreducible component of the preimage
\(h_{j,t'}^{-1}(\widehat C)\) which dominates \(\widehat C\),
and if $C^S\not\subset W_{j,t'}$, then \(C^S\) is invariant by
\(\aF_{j,t'}\). More precisely, at a general smooth point \(p\in C^S\), we have \(T_{C^S,p}\subset T_{\aF_{j,t'},p}\).
\end{lem}

\begin{proof}
Let \(p\in\overline C\) be a general point, and let \(x:=\eta_C(p)\in C\). Then \(x\notin Z(\omega_t)\),  \(x\) is a smooth point of \(C\), and \(\gamma_C(p)=(x,T_{C,x})\). By the definition of the spectral divisor, \(\gamma_C(p)\in D_{\omega_t}\) if and only if \(\omega_t(x)\) vanishes on the line \(T_{C,x}\).
Consequently, \(C\) is invariant by \(\omega_t\) if and only if a
dense open subset of \(\widehat C\) is contained in \(D_{\omega_t}\).
Since \(D_{\omega_t}\) is closed, this is equivalent to \(\widehat C\subset D_{\omega_t}\).

Suppose now that \(\widehat C\subset D_{\omega_t}\). Let \(t'\in\cP^\circ\) such that \(u(t')=t\). We have isomorphisms \(X_{t'}^\circ\simeq X_t,
\cQ_{t'}^\circ\simeq\bP(T_{X_t}),
D_{t'}^\circ\simeq D_{\omega_t}\). Let  \(E\subset D_{t'}^\circ\) be an irreducible component containing \(\widehat C\). We claim that \(E\) dominates
\(X_{t'}^\circ\).

Assume otherwise. The image
\(\Pi_{t'}^\circ(E)\) is then a proper irreducible closed subset of
the surface \(X_{t'}^\circ\). Since it contains \(\Pi_{t'}^\circ(\widehat C)=C\), we have \(\Pi_{t'}^\circ(E)=C\). Therefore \(E\subset(\Pi_{t'}^\circ)^{-1}(C)\). As both \(E\) and \((\Pi_{t'}^\circ)^{-1}(C)\) are integral surfaces, we have \(E=(\Pi_{t'}^\circ)^{-1}(C)\). It follows that, for a general point \(x\in C\), the fiber \((\Pi_{t'}^\circ)^{-1}(x)\simeq\bP(T_{X_t,x})\) is contained in \(D_{t'}^\circ\), so in particular \(\omega_t(x)=0\). Thus a dense open subset of \(C\) is contained in \(Z(\omega_t)\), and hence \(C\subset Z(\omega_t)\), contrary to the hypothesis.

Now we know that the component \(E\) dominates \(X_{t'}^\circ\). By
Lemma~\ref{lem:horizontal-spectral-divisor-family}, there exists an
index \(1\leq j\leq s\) such that \(E=Y_{j,t'}^\circ,\widehat C\subset Y_{j,t'}^\circ\). 
We next show that \(\widehat C\) is tangent to the contact
distribution. 
At a general point \(p\in\overline C\), for every
\(v\in T_{\overline C,p}\) one has
\[
d\Pi_t\bigl(d\gamma_C(v)\bigr)
=
d\eta_C(v)
\in\ell_p.
\]
By the definition of the contact structure, we get \(\boldsymbol\theta_{t'}
\bigl(d\gamma_C(v)\bigr)=0\). 
Since \(\gamma_C\) is an isomorphism over a dense open subset of \(\widehat C\), we deduce that, at a general smooth point of \(\widehat C\), there is an inclusion \(T_{\widehat C}\subset\ker(\boldsymbol\theta_{t'})\).

Let \(C^S\) be an irreducible component of \(h_{j,t'}^{-1}(\widehat C)\) which dominates \(\widehat C\), and assume that \(C^S\not\subset W_{j,t'}\). Let \(p\in C^S\setminus W_{j,t'}\) be a general smooth point such that \(h_{j,t'}(p)\) is a smooth point of \(\widehat C\) and the differential of $h_{j,t'}\vert_{C^S}$ is nonzero. Since \(\boldsymbol\varphi_{j,t'}\) is nonzero at \(p\), it is a
surjective morphism of bundles on a neighborhood of \(p\). Its kernel is
therefore a line subbundle locally around \(p\), which by Lemma~\ref{lem:characteristic-foliation-base-change} is \(T_{\aF_{j,t'}}\).
We have
\[
\boldsymbol\varphi_{j,t'}(T_{C^S,p})
=\boldsymbol\theta_{t'}\bigl(dh_{j,t'}(T_{C^S,p})\bigr)=0.
\]
Thus \(T_{C^S,p}\subset T_{\aF_{j,t'},p}\), which proves that \(C^S\) is invariant by \(\aF_{j,t'}\).
\end{proof}

\subsection{Degrees}\label{sec:family-degree}

Recall that $u:\cP^\circ\to\cP$ denotes the finite etale base change constructed earlier in this section, and that $\cX^\circ:=\cX\times_{\cP}\cP^\circ$. We denote the pullback of the polarization $\aH$ to $\cX^\circ$ by $\aH_{\cX}$. For $t'\in\cP^\circ$, its restriction to $X_{t'}^\circ\simeq X_{u(t')}$ is denoted $H_{u(t')}$.

For \(1\leq j\leq s\), recall the morphisms
\[
q_j:\cS_j\longrightarrow\cP^\circ,
\qquad
r_j:\cS_j\longrightarrow\cX^\circ.
\]
Since \(q_j\) is projective, we can choose a \(q_j\)-ample line bundle
\(\mathcal A_j\) on \(\cS_j\). Define
\[
\aH_j
:=
\mathcal A_j\otimes r_j^*\mathcal H_{\cX}.
\]
For \(t'\in\cP^\circ\), denote the fiber restrictions by
\[
A_{j,t'}:=\mathcal A_j|_{S_{j,t'}},
\qquad
H_{j,t'}:=\aH_j|_{S_{j,t'}}.
\]
Then \(A_{j,t'}\) is ample and
\(r_{j,t'}^*H_{u(t')}\) is nef. Hence \(H_{j,t'}\) is ample.

\begin{lem}\label{lem:two-sided-comparison-fixed-L}
For every \(1\leq j\leq s\), every \(t'\in\cP^\circ\), and every integral curve \(\Gamma\subset S_{j,t'}\), one has \(\bigl(r_{j,t'}^*H_{u(t')}\bigr)\cdot\Gamma
\leq H_{j,t'}\cdot\Gamma\). \end{lem}

\begin{proof}
By the definition of \(\aH_j\), we have \(H_{j,t'}
=
A_{j,t'}\otimes r_{j,t'}^*H_{u(t')}\). Therefore
\[
H_{j,t'}\cdot\Gamma
=
A_{j,t'}\cdot\Gamma
+
\bigl(r_{j,t'}^*H_{u(t')}\bigr)\cdot\Gamma.
\]
Since \(A_{j,t'}\) is ample, \(A_{j,t'}\cdot\Gamma>0\), and the assertion follows.
\end{proof}

\begin{lem}
\label{lem:uniform-resolution-exceptional-locus}
Up to replacing \(\cP\) by a nonempty Zariski-open subset and \(\cP^\circ\) by its inverse image, there is a closed subscheme \(\cN\subset\cX\) 
whose fibers over \(\cP\) have dimension at most one and which satisfies the following properties. First, there is a constant \(\tau_{\mathrm{exc}}>0\), independent of
\(t\in\cP\), such that every integral curve
\(C\) contained in the scheme-theoretic fiber \(\cN_t\) satisfies
\[
 H_t\cdot C\leqslant\tau_{\mathrm{exc}}.
\]
Second, if \(C\subset X_t\) is a \(\omega_t\)-invariant curve with \(C\not\subset Z(\omega_t)\cup\cN_t\), then for every \(t'\in\cP^\circ\) above \(t\), one may choose the curve \(C^S\subset S_{j,t'}\) in Lemma~\ref{lem:invariant-lift} so that \(C^S\to C\) is birational, and 
\[
 g(C^S)=g(C),
 \qquad
 H_t\cdot C\leqslant H_{j,t'}\cdot C^S.
\]
\end{lem}

\begin{proof}
For each \(1\leq j\leq s\), let
\(U_j\subset\cY_j^\circ\) be the largest open subset over which \(h_j:\cS_j\to\cY_j^\circ\) is locally an isomorphism, and let \(B_j:=\cY_j^\circ\setminus U_j\) be the complement. The proof of Lemma~\ref{lem:relative-resolution-family} shows, up to
shrinking $\cP^\circ,\cP$, that
\begin{equation}\label{eq:dimension_Bjt_lessthan_one}
\dim (B_j)_{t}\leq 1, \quad \text{for every}\, t\in \cP^\circ.
\end{equation}
Recall that the $\mathcal W_j$ were introduced at the beginning of \S~\ref{subsec:invariant-curves}. The following union, which is projective over \(\cP^\circ\),
\[
 \cN^\circ := \bigcup_{j=1}^s
 \bigl(\Pi^\circ(B_j)\cup r_j(\mathcal W_j)\bigr) \subset\cX^\circ
\]
has fibers of dimension at most one. Let \(\cN\) be the image of \(\cN^\circ\) in \(\cX\). We have \(\dim\cN_t\leq 1\) for every \(t\in \cP\).

Up to shrinking \(\cP\) we can and will assume that every irreducible component of \(\cN\) dominates \(\cP\). By upper semicontinuity of fiber dimension, we can shrink further \(\cP\) so that each irreducible component with zero-dimensional generic fiber has only fibers of dimension zero. Let \(\cN^{(1)}\) be the union of the remaining components whose generic fibers have dimension one. We are going to define a constant \(\tau_{\mathrm{exc}}\). If \(\cN^{(1)}\) is empty, we set simply \(\tau_{\mathrm{exc}}:=1\). Otherwise we define this constant as follows. Up to shrinking we can assume that \(\cN^{(1)}\to\cP\) is flat. The Hilbert polynomial of the fiber \(\cN_t^{(1)}\), with respect to some chosen very ample power \(H_t^{\otimes k}\), is therefore independent of \(t\). We define \(\tau_{\mathrm{exc}}\) to be the coefficient of its linear term. Every integral curve \(C\subset\cN_t\) occurs with positive multiplicity in the one-dimensional cycle of the fiber \(\cN_t^{(1)}\), so \(k(H_t\cdot C)\leq\tau_{\mathrm{exc}}\), and hence \(H_t\cdot C\leq\tau_{\mathrm{exc}}\).

Now let \(C\not\subset Z(\omega_t)\cup\cN_t\). By
Lemma~\ref{lem:invariant-lift}, for every \(t'\in \cP^\circ\) above \(t\in \cP\) there is an index \(j\) such that \(\widehat C\subset Y_{j,t'}^\circ\). By the definition of \(\cN\) we have that \(\widehat C\not\subset(B_j)_{t'}\).  Hence the closure
\[
 C^S
 :=
 \overline{h_{j,t'}^{-1}
 \bigl(\widehat C\cap(U_j)_{t'}\bigr)}
 \subset S_{j,t'}
\]
is the strict transform of \(\widehat C\), and \(C^S\to\widehat C\) is birational. Since \(C\not\subset\cN_t\) and \(r_j(\mathcal W_j)\subset\cN^\circ\), we also have \(C^S\not\subset W_{j,t'}\). Thus we deduce from Lemma~\ref{lem:invariant-lift} that \(C^S\) is \(\aF_{j,t'}\)-invariant. 

By Lemma~\ref{lem:invariant-lift}
we also know that \(C^S\) is \(\aF_{j,t'}\)-invariant. As the composition \(C^S\to C\) is birational, their geometric genera coincide. Finally, by the projection formula and by Lemma~\ref{lem:two-sided-comparison-fixed-L} we get \(H_t\cdot C =\bigl(r_{j,t'}^*H_t\bigr)\cdot C^S  \leq H_{j,t'}\cdot C^S\).
\end{proof}

The following summarizes what we did in this section:
\begin{cor}\label{cor:degree-control}
Let \(t\in\cP\), and let \(C\subset X_t\) be an
\(\omega_t\)-invariant integral curve such that \(C\not\subset Z(\omega_t)\cup\cN_t\). 
For every \(t'\in\cP^\circ\) with \(u(t')=t\), there exist an index \(1\leq j\leq s\) and an integral curve \(C^S\subset S_{j,t'}\) such that \(C^S\) is \(\aF_{j,t'}\)-invariant, the morphism \(C^S\to C\) is birational, and
\[
g(C^S)=g(C),
\qquad
H_t\cdot C
\leq
H_{j,t'}\cdot C^S.
\]
\end{cor}

Finally we make an observation needed later:
\begin{lem}\label{lem:resolved-surfaces-general-type}
Assume that every fiber $X_t$ is of general type. Then every surface
$S_{j,t'}$ is of general type.
\end{lem}

\begin{proof}
By Lemmas~\ref{lem:horizontal-spectral-divisor-family} and
\ref{lem:relative-resolution-family}, the morphism
$r_{j,t'}:S_{j,t'}\to X_{u(t')}$ is dominant and generically finite.
A smooth projective surface admitting a generically finite dominant
morphism to a surface of general type is of general type.
\end{proof}

\section{From relative two-jets to families of multi-foliations}
\label{sec:multi-foliation}

The goal of this section is to explain how to construct a relative twisted symmetric differential form $\boldsymbol{\omega}$, which serves as the initial input for Section~\ref{sec:family-hypersurfaces}, using techniques from the theory of jet differentials. Since most preparatory materials of this section can be found in the literature, we keep the exposition brief and refer the reader to other references for further details (see for example \cite{Demailly-ElGoul2000,Hou-Huynh-Merker-Xie2026}).

\subsection{The Demailly--Semple tower}\label{subsec:DS-tower}

The conceptual foundation of jet differential techniques traces back to Bloch~\cite{Bloch1926}, who introduced the idea of using differential equations to constrain entire holomorphic curves. This insight was later developed by Green and Griffiths~\cite{Green-Griffiths1980}, who formulated the theory of jet differentials and proposed a fundamental vanishing principle: every negatively twisted jet differential $\omega$ on a projective manifold $X$ should yield an algebraic differential equation $f^*\omega\equiv 0$ satisfied by every entire holomorphic curve $f: \mathbb{C}\rightarrow X$. This principle provides a link between the existence of negatively twisted jet differentials and hyperbolicity: when sufficiently many such differentials exist, they force every entire curve to be algebraically degenerate. The vanishing theorem was subsequently refined and rigorously proved by Siu--Yeung~\cite{Siu-Yeung1996a}. Meanwhile, Demailly introduced \emph{invariant} jet differentials~\cite{Demailly1997}, which have proved particularly effective in the study of hyperbolicity on surfaces (see \cite{Demailly-ElGoul2000, Paun2008, Hou-Huynh-Merker-Xie2025, Hou-Huynh-Merker-Xie2026}).

For a smooth projective surface $X_t$, the Demailly--Semple tower~\cite{Demailly1997} is defined by successive projectivizations
\[
X_{t,2} \xrightarrow{\pi_{2,1}} X_{t,1} \xrightarrow{\pi_{1,0}} X_{t,0}:=X_t,
\]
where $X_{t,1}=\mathbb{P}(T_{X_t})$ and
$X_{t,2}=\mathbb{P}(V_{t,1})$ with
\[
V_{t,1,(x,[v])} = \{\xi\in T_{X_{t,1},(x,[v])}\mid (\pi_{1,0})_*\xi\in\mathbb{C}\cdot v\}.
\]
One can think of sections of $\mathcal{O}_{X_{t,2}}(m)$ as invariant $2$-jet differentials of weighted degree $m$. If
\(f:\bC\to X_t\) is a nonconstant holomorphic map, then its first lift \(f_{[1]}\) is the holomorphic extension of \(z\mapsto(f(z),[df_z])\) across the critical points of \(f\) while its second lift \(f_{[2]}\) is defined similarly.

Fix an integer \(d\geq 18\), and let \(p_d:\cX_d\to\mathcal H_d^{\mathrm{sm}}\) be the universal family of
smooth degree-\(d\) surfaces. Let \(\aO_{\cX_d}(1):=\pr_1^*\aO_{\bP^3}(1)|_{\cX_d}\), where
\(\pr_1:\bP^3\times\mathcal H_d^{\mathrm{sm}}\to\bP^3\) is the first projection. 
Denote by
\[
\cX_{d,2}\xrightarrow{\pi_{2,1}}\cX_{d,1}
=\bP\bigl(T_{\cX_d/\mathcal H_d^{\mathrm{sm}}}\bigr)
\xrightarrow{\pi_{1,0}}\cX_d
\]
the first two levels of the relative Demailly--Semple tower.

\subsection{Construction of the symmetric differential}
\label{sec:jet-to-symmetric}

The following result, which relies on the existence of two independent invariant \(2\)-jet differentials, can be deduced from \cite{Demailly-ElGoul2000,Paun2008} and is fully proved in \cite{Hou-Huynh-Merker-Xie2026}\footnote{\cite{Hou-Huynh-Merker-Xie2026} and \cite{CuiHouLiuXie2026} improved the required degree bound from \(18\) to \(16\)}:
\begin{thm}[{\cite[Th.~1.5]{Hou-Huynh-Merker-Xie2026}}]\label{prop:jet-direction-locus}
Let \(d\geq 18\). 
There exist a nonempty Zariski-open subset
\(\cP_d\subset\mathcal H_d^{\mathrm{sm}}\) and a Zariski-closed subvariety
\[
\cS_d\subset
\bP\bigl(T_{\cX_d/\mathcal H_d^{\mathrm{sm}}}\bigr)|_{\cP_d}
\]
such that \(\cS_d\to\cP_d\) is flat and projective, every fiber has
dimension at most \(2\), and
\(f_{[1]}(\bC)\subset(\cS_d)_t\) for every \(t\in\cP_d\) and every
nonconstant holomorphic map \(f:\bC\to X_t\).
\end{thm}

\begin{cor}\label{cor:degree-ge18-relative-symmetric-differential}
Let \(d\geq 18\). Up to replacing the open subset \(\cP_d\) in Theorem~\ref{prop:jet-direction-locus} by a nonempty Zariski-open
subset, there exist integers \(q_d,r_d\geqslant1\) and a nonzero section
\[
\boldsymbol\omega_d\in
\coh^0 \Big(
\cX_d\vert_{\cP_d},
\Sym^{q_d}\Omega^1_{\cX_d\vert_{\cP_d}/\cP_d}
\otimes\aO_{\cX_d}(r_d)\vert_{\cX_d\vert_{\cP_d}}
\Big)
\]
such that, for every \(t\in\cP_d\), the restriction
\(\omega_{d,t}:=\boldsymbol\omega_d|_{X_t}\) is nonzero and
\(f^*(\omega_{d,t})=0\) for every holomorphic map \(f:\bC\to X_t\).
\end{cor}

\begin{proof}
Let \(\eta\) be the generic point of \(\cP_d\). Take an integer \(a\geq 1\) such that
\[
\mathcal A:=
\left(
\aO_{\cX_{d,1}}(1)\otimes
\pi_{1,0}^*\aO_{\cX_d}(a)
\right)
\vert_{\cX_{d,1}\vert_{\cP_d}}
\]
is ample relative to \(\cP_d\). The generic fiber
\((\cX_{d,1})_\eta=\bP(T_{X_\eta})\) has dimension \(3\), whereas
\(\dim(\cS_d)_\eta\leqslant2\). Thus the ideal sheaf \(\mathcal I_{(\cS_d)_\eta}\) is not trivial. For \(N\gg 0\), Serre's theorem says that
\(\mathcal I_{(\cS_d)_\eta}\otimes\mathcal A_\eta^N\) is globally
generated, and hence admits a nonzero section \(s_\eta\).

Up to shrinking \(\cP_d\), the section \(s_\eta\) induces a global section 
\[
\boldsymbol s\in
\coh^0(\cX_{d,1}\vert_{\cP_d},
\mathcal I_{\cS_d}\otimes\mathcal A^N).\]
Because of the natural inclusion
\(\mathcal I_{\cS_d}\hookrightarrow
\aO_{\cX_{d,1}\vert_{\cP_d}}\), we can think of \(\boldsymbol s\) as a section of \(\mathcal A^N\). By construction its nonvanishing locus meets the generic fiber. Since the smooth morphism
\(\cX_{d,1}\vert_{\cP_d}\to\cP_d\) is open, the image of this locus is a nonempty Zariski-open subset of \(\cP_d\). Up to shrinking to this subset, \(\boldsymbol s\) is nonzero on every fiber.

Projecting onto $\cX_d$ we get an isomorphism
\[
(\pi_{1,0})_*\mathcal A^N
\simeq
\Sym^N\Omega^1_{(\cX_d\vert_{\cP_d})/\cP_d}
\otimes
\aO_{\cX_d}(aN)\vert_{\cX_d\vert_{\cP_d}}.
\]
Then \(\boldsymbol s\) gives rise to the desired section
\(\boldsymbol\omega_d\) in the statement with \(q_d=N\) and
\(r_d=aN\). By construction on each fiber the restricted section \(\omega_{d,t}\) is nonzero.

Let \(f:\bC\to X_t\) be nonconstant. By
Theorem~\ref{prop:jet-direction-locus}, we have 
\(f_{[1]}(\bC)\subset(\cS_d)_t\). Since \(\boldsymbol s\) vanishes
along \(\cS_d\), the section \(f^*(\omega_{d,t})\) vanishes on the
dense open subset where \(df\neq0\), and hence vanishes identically.
The conclusion is immediate for constant maps because \(q_d\geq 1\).
\end{proof}

\section{A Poincar\'e Bound for Foliation-Invariant Curves}
\label{sec:poincare}

Let \(S\) be a smooth complex projective surface and let \(\mathcal F\) be a saturated rank-one holomorphic foliation on \(S\). 
Denote by \(T_{\mathcal F}\subset T_S\) its tangent line bundle and by \(N_{\mathcal F}\) its normal line bundle. There is an exact sequence of sheaves (see \cite[\S~2]{BrunellaBook})
\[
0\longrightarrow T_{\mathcal F}\longrightarrow T_S \longrightarrow \mathcal I_Z\otimes N_{\mathcal F}\longrightarrow 0
\]
where \(\mathcal I_Z\) is the ideal sheaf of the zero-dimensional scheme \(Z=\Sing(\mathcal F)\).

Let \(H\) be an ample divisor on \(S\). We will use the notations
\begin{align}
& c:=c_2(S),\quad s:=K_S^2,\quad h:=H^2, \quad \kappa:=K_S\cdot H,\quad
\eta:=N_{\mathcal F}\cdot H,
\label{eq:many-notations-poincare}
\\
& M_0 := c+\frac{\kappa\eta+\eta^2}{h}
+\frac14\left(\frac{\kappa^2}{h}-s\right), \quad
M:=\max\{0,\lceil M_0\rceil\}. \notag 
\end{align}

For a reduced curve \(C\) on \(S\), let \(n_{\mathrm{ord}}(C)\) denote the number of
its ordinary double points and ordinary triple points. For an integral \(\mathcal F\)-invariant curve \(C\) of geometric genus \(g\), we denote
\begin{equation}\label{eq:notations-curve-AB}
A(g):=\max\{0,3c-s+4g-4+M\},\; 
B(g):=\max\{0,3c-s+2g-2\}.
\end{equation}

This section is devoted to the proof of the following result:
\begin{thm}\label{thm:adjoint-Poincare-degree-bound}
Let \(S\) be a smooth complex projective surface of general type, and let
\(\mathcal F\) be a saturated rank-one holomorphic foliation on \(S\).
Let \(C\subset S\) be an \(\mathcal F\)-invariant integral curve of geometric genus \(g\). Then
\[
H\cdot C \leq 
\kappa A(g)+ \sqrt{(\kappa^2-h)(A(g)^2+B(g))}.
\]
\end{thm}

\subsection{Singularities of invariant curves}

If $a\,dx+b\,dy$ is a local form defining $\aF$ around $p\in S$, then the Milnor number is defined as
\[
\mu_p:=\operatorname{length}\bigl(\mathcal{O}_{S,p}/(a,b)\bigr)=\dim_{\mathbb{C}}\bigl(\mathcal{O}_{S,p}/(a,b)\bigr).
\]
By construction, $\sum_{p\in \Sing(\aF)}\mu_p=\deg Z$. We will denote \(E_{\mathcal F}:=\Omega^1_S\otimes N_{\mathcal F}\). 
\begin{prop}\label{prop:poincare-sing-bound}
Let \(\mathcal F\) be a saturated rank-one foliation on \(S\). Then
\[
\sum_{p\in\Sing(\mathcal F)}
\mu_p(\mathcal F)
=
c_2(E_{\mathcal F})
=
c_2(S)+K_S\cdot N_{\mathcal F}+N_{\mathcal F}^2 .
\]
The number of singular points of \(\mathcal F\) is bounded above by \(M\) from \eqref{eq:many-notations-poincare}.
\end{prop}

\begin{proof}
The formula \(\sum_p\mu_p(\mathcal F)=c_2(E_{\mathcal F})\) follows from \cite[Proposition 14.1]{FultonBook} (see \cite[Section 2.1]{BrunellaBook}). In particular, the number of singular points of \(\mathcal F\) is bounded above by \(c_2(E_{\mathcal F})=c_2(S)+K_S\cdot N_{\mathcal F}+N_{\mathcal F}^2\). It remains to estimate this number. Consider the decomposition, orthogonal with respect to the intersection pairing:
\[
K_S=\frac{\kappa}{h}H+K_0,
\qquad
N_{\mathcal F}=\frac{\eta}{h}H+N_0 ,
\]
where $K_0,N_0$ satisfy \(K_0\cdot H=N_0\cdot H=0\). 
By the Hodge index theorem, we have \(\left(N_0+\frac{1}{2}K_0\right)^2\leq 0\). Thus
\[
K_0\cdot N_0+N_0^2
\leqslant-\frac14K_0^2 .
\]
Since 
\[
K_0^2=\left(K_S-\frac{\kappa}{h}H\right)^2=s-\frac{\kappa^2}{h},
\]
we obtain
\[
c_2(E_{\mathcal F})
\leqslant
c+\frac{\kappa\eta+\eta^2}{h}
+\frac14
\left(
\frac{\kappa^2}{h}-s
\right)
=M_0 .
\]
\end{proof}

\begin{cor}\label{cor:poincare-curve-sing}
Let \(C\subset S\) be a reduced \(\mathcal F\)-invariant curve. Then \(\Sing(C)\subseteq\Sing(\mathcal F)\). In particular, \(n_{\mathrm{ord}}(C)\leqslant M\). 
\end{cor}

\subsection{Lu--Miyaoka inequalities}

The following inequalities are proved in the paper of Lu--Miyaoka~\cite[Th.~1]{LuMiyaoka}.

\begin{thm}[Lu--Miyaoka]\label{thm:LuMiyaoka}
Let \(S\) be a smooth projective surface and let
\(C\subset S\) be an irreducible curve of geometric genus \(g\).
Assume that some positive multiple of \(K_S+C\) is effective. Then
\begin{align*}
K_S\cdot C & \leq 
3c_2(S)-K_S^2+4g-4+n_{\mathrm{ord}}(C),
\\
C^2 & \geq 
K_S^2-3c_2(S)-2g+2+n_{\mathrm{ord}}(C).
\end{align*}
\end{thm}

Recall that in this section we work with a smooth projective surface \(S\) equipped with a holomorphic foliation \(\mathcal F\), and that we have fixed some notations in \eqref{eq:many-notations-poincare}, \eqref{eq:notations-curve-AB}. 
\begin{lem}\label{lem:poincare-canonical-bound}
Assume that \(S\) is of general type. Let \(C\subset S\) be an integral \(\mathcal F\)-invariant curve of geometric genus \(g\). Then \(K_S\cdot C\leq A(g), C^2\geq -B(g)\).
\end{lem}

\begin{proof}
Since \(S\) is of general type, there is an integer \(r>0\) such that \(rK_S\) is effective. Then \(r(K_S+C)=rK_S+rC\) is also effective. Therefore
Theorem~\ref{thm:LuMiyaoka} applies. By Corollary~\ref{cor:poincare-curve-sing}, we get \(n_{\mathrm{ord}}(C)\leq M\). 
The first inequality of Theorem~\ref{thm:LuMiyaoka} implies
\[
K_S\cdot C \leq 3c-s+4g-4+M ,
\]
and hence \(K_S\cdot C\leq A(g)\). The second inequality of Theorem~\ref{thm:LuMiyaoka} implies
\[
C^2 \geq s-3c-2g+2+n_{\mathrm{ord}}(C) \geq -B(g).
\]
\end{proof}

We now convert the bound on the intersection with the canonical bundle into a bound
of \(H\)-degree.

\begin{lem}\label{lem:poincare-hodge-bound}
Let \(S\) be a smooth projective surface of general type, and \(H\) be an ample line bundle. Let \(A,B\geq 0\). Suppose that an integral curve \(C\subset S\) satisfies \(K_S\cdot C\leq A,
C^2\geq-B\), Then
\[
H\cdot C \leq
(K_S\cdot H)A +
\sqrt{
\bigl((K_S\cdot H)^2-H^2\bigr)(A^2+B)
}.
\]
\end{lem}

\begin{proof}
Let \(\pi:S\rightarrow S_{\min}\) be the blow-down morphism onto the minimal model and let \(P=\pi^*K_{S_{\min}}\) be the pull-back of the canonical bundle. Then \(P\) is nef and big. In particular, \(P^2>0\). We have \(K_S=P+E\) where \(E\) is the exceptional curve. Thus,
\begin{equation}\label{eq:Hdegree-Kdegree-inequality-1}
P\cdot H\leq K_S\cdot H,
\quad
P\cdot C = K_S\cdot C-E\cdot C.
\end{equation}
If \(C\) is not \(\pi\)-exceptional, then \(E\cdot C\geq 0\). Otherwise \(P\cdot C= 0\). Hence 
\begin{equation}\label{eq:Kdegree-Hdegree}
0\leq P\cdot C \leq A.
\end{equation}
Consider the orthogonal decomposition 
\[
H=\alpha P+H_0,
\qquad
C=\beta P+C_0
\]
where \(\alpha=P\cdot H/ P^2\), \(\beta= P\cdot C/P^2\), and \(H_0\cdot P=C_0\cdot P=0\). 
By the Hodge index theorem the inteserction pairing is negative definite on the orthogonal of \(P\). Thus we get
\begin{equation}\label{eq:Hodge-index-nonnegative-H0C0}
H_0\cdot C_0
\leq \sqrt{(-H_0^2)(-C_0^2)}.
\end{equation}
Then we compute:
\begin{align*}
H\cdot C & = (\alpha P+H_0)\cdot (\beta P+C_0)\\
& \leq
\frac{(P\cdot H )(P\cdot C)}{P^2}
+
\sqrt{
\left(\frac{(H\cdot P)^2}{P^2}-H^2\right)
\left(\frac{(P\cdot C)^2}{P^2}-C^2\right)}.
\end{align*}
Then using \eqref{eq:Hdegree-Kdegree-inequality-1},  \eqref{eq:Kdegree-Hdegree}, the hypothesis \(C^2\geq -B\) and the positivity \(P^2\geq 1\), we obtain
\begin{align*}
H\cdot C & \leq
(K_S\cdot H)A+
\sqrt{\bigl((K_S\cdot H)^2-H^2 \bigr)(A^2+B)}.
\end{align*}
\end{proof}

Now we are ready to prove Theorem~\ref{thm:adjoint-Poincare-degree-bound}.
\begin{proof}[Proof of Theorem~\ref{thm:adjoint-Poincare-degree-bound}]
Let \(C\) be an integral \(\mathcal F\)-invariant curve of geometric genus \(g\). By Lemma~\ref{lem:poincare-canonical-bound} we have
\[
K_S\cdot C\leq A(g),
\quad
C^2\geq-B(g) .
\]
By Lemma~\ref{lem:poincare-hodge-bound}, we obtain
\[
H\cdot C
\leqslant
\kappa A(g)+
\sqrt{(\kappa^2-h)(A(g)^2+B(g))} .
\]
\end{proof}

\section{Main proofs for hypersurfaces}
\label{sec:main-proofs-hypersurfaces}

\subsection{Bounding degrees of invariant curves}
\label{subsec:bounding-degrees}

Let $\cP$ be a smooth variety, let $p:\cX\to\cP$ be a
smooth projective morphism with smooth surface fibers of
general type, and let $\aH$ be a $p$-ample line bundle. For $t\in\cP$, denote $H_t:=\aH\vert_{X_t}$. 
\begin{thm}\label{thm:main-degree-bound}
Let $m\geq 1$. Let $\aL$ be a line bundle on $\cX$. Consider a nonzero section
\[
\boldsymbol\omega\in
\coh^0\bigl(\cX,\Sym^m\Omega^1_{\cX/\cP}\otimes\mathcal L\bigr).
\]
For every integer $g\geq 0$, there are an integer $\tau\geq 1$ and a
nonempty Zariski-open subset $\cP_g\subset\cP$ with the following
property. If $t\in\cP_g$ and $C\subset X_t$ is an integral curve of
geometric genus $g$ such that either $C\subset Z(\omega_t)$, or
$C\not\subset Z(\omega_t)$ and $C$ is $\omega_t$-invariant, then
\(H_t\cdot C\leq\tau\).
\end{thm}

\begin{proof}
We perform the reduction procedure explained in
Section~\ref{sec:family-hypersurfaces}. In this proof all assertions hold only up to shrinking $\cP$. Lemma~\ref{lem:uniform-base-locus-degree} provides a constant \(\tau_Z\geq 0\) such that \(H_t\cdot C\leq\tau_Z\) for every integral curve \(C\subset Z(\omega_t)\). So we need to deal mainly with the case \(C\not\subset Z(\omega_t)\).

For each \(j\), every
\(S_{j,t'}\) is of general type by Lemma~\ref{lem:resolved-surfaces-general-type}. The numerical data \(c,s,h,\kappa,\eta\) defined in \eqref{eq:many-notations-poincare}, computed with respect to the polarizations \(H_{j,t'}\) and the foliations \(\aF_{j,t'}\) from Section~\ref{sec:family-hypersurfaces} where $t'\in \cP^\circ$, are locally constant on \(\cP^\circ\). This is immediate for \(c,s,h,\kappa\) because the family \(q_j:\cS_j\to\cP^\circ\) is smooth. By
Lemma~\ref{lem:characteristic-foliation-base-change}, the relative tangent sheaf \(T_{\cF_j}\subset T_{\cS_j/\cP^\circ}\) restricts to
\(T_{\aF_{j,t'}}\) on every fiber. Since \(\cS_j\) is smooth and \(T_{\cF_j}\) is saturated of rank one, it is a line bundle.
Moreover, \(K_{\cS_j/\cP^\circ}^{-1}\otimes T_{\cF_j}^{-1}\) restricts to \(N_{\aF_{j,t'}}\). Hence the intersection number
\(N_{\aF_{j,t'}}\cdot H_{j,t'}\) is locally constant.

As there are finitely many indices \(j\) and finitely many connected components of \(\cP^\circ\), 
Theorem~\ref{thm:adjoint-Poincare-degree-bound} yields a constant \(\tau_{\mathrm{fol}}>0\) such that
\(H_{j,t'}\cdot\Gamma\leq\tau_{\mathrm{fol}}\) for every integral \(\aF_{j,t'}\)-invariant curve
\(\Gamma\subset S_{j,t'}\) of geometric genus \(g\).

Now let \(C\subset X_t\) be an integral curve of geometric genus \(g\) for some $t\in \cP$. If
\(C\subset Z(\omega_t)\), then we already have \(H_t\cdot C\leq\tau_Z\). If
\(C\subset\cN_t\), Lemma~\ref{lem:uniform-resolution-exceptional-locus}
says that \(H_t\cdot C\leq\tau_{\mathrm{exc}}\). Otherwise,
Corollary~\ref{cor:degree-control} produces an integral
\(\aF_{j,t'}\)-invariant curve \(C^S\subset S_{j,t'}\) of geometric
genus \(g\) such that \(H_t\cdot C\leq H_{j,t'}\cdot C^S\leq\tau_{\mathrm{fol}}\). 
The assertion follows with $\tau:=\max\{\tau_Z,\tau_{\mathrm{exc}},\tau_{\mathrm{fol}}\}$.
\end{proof}

\subsection{Hilbert schemes}\label{subsec:hilbert-scheme}

\begin{proof}[Proof of Theorem~\ref{mainthm:main-algebraic-degeneracy}]
We fix a relatively ample line bundle $\aH$ on $\cX$.

Let \(t\in\cP\), and let \(C\subset X_t\) be an integral curve of
geometric genus zero or one. Denote its normalization morphism by
\(\eta_C:\overline C\to C\hookrightarrow X_t\). There exists a
nonconstant holomorphic map \(\pi:\bC\to\overline C\) with
Zariski-dense image. The entire curve
\(f:=\eta_C\circ\pi:\bC\to X_t\) satisfies \(f^*(\omega_t)=0\) by the
hypothesis of the theorem. 
If \(C\not\subset Z(\omega_t)\), then 
\(0=f^*(\omega_t)\) implies that the curve \(C\) is \(\omega_t\)-invariant. Thus every integral curve of geometric genus zero
or one satisfies one of the two alternatives in
Theorem~\ref{thm:main-degree-bound}.

Since \(\cP\) is irreducible, applying
Theorem~\ref{thm:main-degree-bound} for \(g=0\) and \(g=1\), we obtain a nonempty Zariski-open subset \(\cU\subset\cP\) and
an integer \(\tau\geq 1\) such that every integral curve of geometric genus zero or one in a fiber over \(\cU\) has degree at most \(\tau\).

Up to shrinking $\cU$, some positive power $\aH^{\otimes k}$ defines
a closed immersion $\cX\hookrightarrow\bP^r\times\cU$ with
$r\geq 3$. Up to replacing \(\tau\) by \(k\tau\) we can measure degrees in \(\bP^r\) because \(\deg_{\bP^r}C=k(H_t\cdot C)\). Let \(\Gamma\subset\bP^r\) be an integral curve of degree
\(e\leq\tau\). Its Hilbert polynomial with variable \(n\), by Riemann--Roch Theorem, is \(en+1-p_a(\Gamma)\) where \(p_a\) denotes the arithmetic genus. A general linear projection induces a finite birational morphism
\(\Gamma\to\Gamma'\) onto an integral plane curve of degree \(e\). From the
inclusion of \(\aO_{\Gamma'}\) into the direct image of
\(\aO_\Gamma\) we deduce that \(0\leq p_a(\Gamma)\leq p_a(\Gamma')=(e-1)(e-2)/2\). Consequently only finitely many possible Hilbert polynomials \(P_1,\ldots,P_N\) can occur for the curve \(\Gamma\).

For each possible polynomial \(P_i\), the locus
\(\mathscr M_i\subset\Hilb^{P_i}(\bP^r)\) in the corresponding Hilbert scheme parametrizing geometrically integral curves is open by
\cite[Th.~12.2.1(x)]{EGAIV}. By \cite[Lem.~2.12 and Cor.~2.13]{DethloffOrevkovZaidenberg},
the base of a family of curves admits a finite locally closed
stratification over which the family has simultaneous normalization, and in particular the locus
\(\mathscr G_i\subset\mathscr M_i\) parametrizing curves of geometric genus zero or one is constructible.

For each \(i\), the condition \(C\subset X_t\) defines a closed subset of
\(\Hilb^{P_i}(\bP^r)\times\cU\). Its intersection with
\(\mathscr G_i\times\cU\) is constructible. Its
image \(\Sigma_i\subset\cU\) is also constructible. Hence
\(\Sigma:=\bigcup_{i=1}^N\Sigma_i\) is constructible and consists exactly
of those \(t\in\cU\) for which \(X_t\) contains an integral curve of
geometric genus zero or one.

By the hypothesis of the theorem, a very general fiber contains no
integral curve of geometric genus zero or one. Therefore
\(\Sigma\) is contained in a countable union of proper closed subsets of
\(\cU\). If \(\Sigma\) were Zariski dense in \(\cU\), then, being
constructible, it would contain a nonempty Zariski-open subset of
\(\cU\). This is impossible because \(\cU\) is an irreducible variety
over the uncountable field \(\bC\), and no nonempty Zariski-open subset
of such a variety is contained in a countable union of proper closed
subsets. Thus \(\overline{\Sigma}\) is a proper closed set of $\cU$, and
\(\cU\setminus\overline{\Sigma}\) is a nonempty Zariski-open subset parametrizing surfaces containing no integral curve of geometric genus
zero or one.
\end{proof}

\begin{proof}[Proof of Theorem~\ref{mainthm:hypersurface-hyperbolicity}]
We shrink the parameter space so that Theorem~\ref{mainthm:main-algebraic-degeneracy} holds and \(\omega_t\neq 0\) for every \(t\in\cP\). Let
\(f:\bC\to X_t\) be a nonconstant entire curve. We claim that \(f\) is algebraically degenerate, i.e.\ has image contained in a proper Zariski-closed set. We have \(f^*(\omega_t)=0\).

Suppose, on the contrary, that \(f(\bC)\) is Zariski dense in \(X_t\).
Fix \(t'\in\cP^\circ\) with \(u(t')=t\). By the definition of the
spectral divisor (see \S~\ref{subsec:spectral-divisor}), we have that \(f_{[1]}(\bC)\subset D_{\omega_t}\). The Zariski closure of
\(f_{[1]}(\bC)\) is irreducible and dominates \(X_t\). Thus, by
Lemma~\ref{lem:horizontal-spectral-divisor-family}, it is one of the
components \(Y_{j,t'}^\circ\). Since \(h_{j,t'}:S_{j,t'}\to Y_{j,t'}^\circ\) is birational, the composition \(h_{j,t'}^{-1}\circ f_{[1]}\) defines a meromorphic map which extends to a holomorphic map \(\widetilde f:\bC\to S_{j,t'}\). Its image is Zariski dense in
\(S_{j,t'}\), and in particular is not contained in the set \(W_{j,t'}\) introduced in \S~\ref{subsec:invariant-curves}. Then \(\widetilde f\) is tangent to the foliation \(\aF_{j,t'}\) (see Lemma~\ref{lem:invariant-lift} and its proof). In orther words, it gives a parabolic leaf of a foliation on the surface \(S_{j,t'}\) which is of general type by Lemma~\ref{lem:resolved-surfaces-general-type}. McQuillan's theorem~\cite{McQuillan1998,McQuillan1999}
(cf.\ \cite[Th.~6.1]{Demailly-ElGoul2000}) therefore implies that
\(\widetilde f\) is algebraically degenerate, a contradiction.
Hence \(f\) is algebraically degenerate.

Thus \(f(\bC)\) is contained in an integral curve \(C\subset X_t\) which has geometric genus at most one. This contradicts Theorem~\ref{mainthm:main-algebraic-degeneracy}. Hence every entire curve in \(X_t\) is constant, and Brody's lemma~\cite{Brody} implies that \(X_t\) is Kobayashi hyperbolic.
\end{proof}

\begin{proof}[Proof of Theorems~\ref{mainthm:kobayashi-hypersurface} and~\ref{mainthm:no-rational-curves}]
Let \(d\geq 18\). Corollary~\ref{cor:degree-ge18-relative-symmetric-differential} provides, over a nonempty Zariski-open subset of
\(\mathcal H_d^{\mathrm{sm}}\), a relative twisted symmetric differential
satisfying the differential hypothesis of Theorem~\ref{mainthm:main-algebraic-degeneracy}. By Clemens--Xu's theorem~\cite{Clemens1986,Xu1994}, a very general degree-$d$ surface contains no rational or elliptic curves. Thus all the hypotheses of Theorem~\ref{mainthm:main-algebraic-degeneracy} hold, and we obtain Theorem~\ref{mainthm:no-rational-curves}. Applying Theorem~\ref{mainthm:hypersurface-hyperbolicity} to the same family, we also obtain Theorem~\ref{mainthm:kobayashi-hypersurface}.
\end{proof}

\section{Complements of two plane curves}
\label{sec:curve-complements}

In this section we fix integers $1\leq d_1\leq d_2$ such that
\begin{equation}\label{eq:d1d2-condition}
d_1,d_2\geq 3,
\quad\text{or}\quad
d_1=2,\ d_2\geq 5,
\quad\text{or}\quad
d_1=1,\ d_2\geq 8.
\end{equation}
Let $\cP_{d_1,d_2}^{\mathrm{snc}}$ be the parameter space defined in
\eqref{eq:plane-pair-parameter-space}, parametrizing simple normal crossing (snc) divisors of $\bP^2$ with two components of degrees $d_1,d_2$. We denote the universal curves over $\cP_{d_1,d_2}^{\mathrm{snc}}$ by $\cC_1$ and $\cC_2$, and define $\cD:=\cC_1\cup\cC_2\subset\bP^2\times\cP_{d_1,d_2}^{\mathrm{snc}}$. For a parameter $a\in \cP_{d_1,d_2}^{\mathrm{snc}}$, we denote the corresponding snc divisor by $D_a=C_{1,a}+C_{2,a}$.  

\subsection{The logarithmic direction locus}\label{subsec:log-locus}

Let $D\subset\bP^2$ be a simple normal crossing divisor, let
$\Gamma\not\subset D$ be an irreducible curve, and let 
$\nu:\widetilde\Gamma\to\Gamma\hookrightarrow \bP^2$ be its normalization. Denote by $B:=\operatorname{Supp}\nu^*D$ the reduced divisor defined by the pullback. The logarithmic differential
$d\nu:T_{\widetilde\Gamma}(-\log B)\to
\nu^*T_{\bP^2}(-\log D)$ is nonzero on a dense Zariski-open subset of $\widetilde\Gamma$. It induces a morphism
\begin{equation}\label{eq:log-tangent-lift}
\nu_{[1]}:
\widetilde\Gamma\to
\bP\bigl(T_{\bP^2}(-\log D)\bigr),
\quad
x\mapsto
\left[
d\nu_x\bigl(T_{\widetilde\Gamma,x}(-\log B)\bigr)
\right]
\end{equation}
that we call \emph{the logarithmic tangent lift of $\Gamma$ or of $\nu$}.

\begin{thm}[Hou--Wang--Xie{\cite[Th.~1.2]{HouWangXie2026}}]\label{thm:plane-pair-direction-locus}
There are a nonempty Zariski-open subset
$\cU\subset\cP_{d_1,d_2}^{\mathrm{snc}}$, and a closed subscheme
$\cS\subset\bP(T_{\bP^2\times\cU/\cU}(-\log\cD))$ which is projective
over $\cU$ and whose fibers $S_a$ all have dimension $\leq 2$ such that the following holds. If $\Gamma\not\subset D_a$ is a rational curve
with normalization $\nu:\bP^1\to\Gamma$ such that $\#\operatorname{Supp}\nu^*D_a\leq 2$, then its logarithmic tangent lift
is contained in $S_a$.
\end{thm}

From now on, we will use $\cU$ and $\cS$ as in Theorem~\ref{thm:plane-pair-direction-locus}. Note that we will replace $\cU$ by a smaller nonempty Zariski-open subset whenever necessary.

\subsection{Cyclic cover}\label{subsec:cyclic-cover}
In this subsection we will carry out the classical cyclic cover (see \cite[\S~3]{EsneaultViehweg} and \cite{Pardini}) construction in family. 
Up to replacing $\cU$ by a nonempty affine open subset, we can and will assume that universal curves are defined by sections
\[
F_i\in\coh^0\bigl(\bP^2\times\cU,
\operatorname{pr}_1^*\aO_{\bP^2}(d_i)\bigr), \quad i=1,2.
\]
We let $\aH_0:=\operatorname{pr}_1^*\aO_{\bP^2}(1)$ be the pullback line bundle on $\bP^2\times \cU$. For $i=1,2$, consider the locally free $\aO_{\bP^2\times\cU}$-module
\[
 \mathcal A_i:=\bigoplus_{j=0}^{d_i-1}\aH_0^{-j}.
\]
We define a structure of commutative algebra on $\mathcal A_i$ as follows. Let $s$ and $t$ be local sections of $\aH_0^{-j}$ and $\aH_0^{-k}$, respectively. Since $0\leq j,k<d_i$, one has $j+k<2d_i$. Their product is defined by
\[
 s\cdot t=
 \begin{cases}
 s\otimes t\in\aH_0^{-(j+k)},
   & j+k<d_i,\\
 F_i\cdot (s\otimes t)
   \in\aH_0^{-(j+k-d_i)},
   & j+k\geq d_i.
 \end{cases}
\]

Define
\begin{equation}\label{eq:plane-pair-cover-algebra}
 q:\cY:=\operatorname{Spec}_{\bP^2\times\cU}
 \bigl(\mathcal A_1\otimes_{\aO_{\bP^2\times\cU}}\mathcal A_2\bigr)
 \longrightarrow\bP^2\times\cU.
\end{equation}
where the tensor product is taken in the category of $\aO_{\bP^2\times\cU}$-algebras. Its underlying module is
locally free of rank $d_1d_2$. By construction we have an isomorphism of $\aO_{\bP^2\times\cU}$-modules:
\begin{equation}\label{eq:plane-pair-cover-direct-image}
 q_*\aO_{\cY}\simeq
 \bigoplus_{\substack{0\leq j<d_1\\0\leq k<d_2}}
 \aH_0^{-(j+k)}.
\end{equation}

The morphism $q$ can be described locally as follows. Let $W\subset\bP^2\times\cU$ be an open subset on which $\aH_0$ is
generated by a nowhere-vanishing section $\alpha$. For $i=1,2$, over $W$, we have $F_i=f_i\alpha^{d_i}$ with $f_i\in\coh^0(W,\aO_W)$. We have an
isomorphism of $\aO_W$-algebras
\[
 \mathcal A_i\vert_W
 \simeq
 \frac{\aO_W[T_i]}{(T_i^{d_i}-f_i)},
\]
so
\begin{equation}\label{eq:qW-spec-ring}
q^{-1}(W) \simeq \operatorname{Spec}_W 
\frac{\aO_W[T_1,T_2]}{(T_1^{d_1}-f_1,T_2^{d_2}-f_2)}.
\end{equation}
One can verify that the local sections $T_iq^*\alpha$ glue to a global section $\tau_i\in\coh^0(\cY,q^*\aH_0)$ such that
\begin{equation}\label{eq:plane-pair-root-sections}
 \tau_i^{d_i}=q^*F_i.
\end{equation}

Let $\cR_i\subset\cY$ be the zero scheme of $\tau_i$, and define
$\cR:=\cR_1+\cR_2$. Let $p:\cY\to\cU$ be the projection morphism. For $a\in\cU$, let $q_a:Y_a\to\bP^2$ be the fiber morphism of $q$, and denote the fibers of $\cR_i$ and $\cR$ by $R_{i,a}$ and $R_a$, respectively.

\begin{prop}\label{prop:plane-pair-cyclic-cover}
The morphism $p$ is smooth and projective of relative dimension two, and
has smooth projective surfaces as fibers. For $i=1,2$, the closed subscheme
$\cR_i\subset\cY$ is an effective Cartier divisor, and
$\cR_i\to\cU$ is smooth of relative dimension one, while the scheme-theoretic intersection $\cR_1\cap\cR_2\to\cU$ is smooth of relative dimension zero, i.e.\ $\cR=\cR_1+\cR_2$ has simple normal crossings relative to $\cU$. For $i=1,2$, we also have
\begin{equation}\label{eq:plane-pair-branch-pullback}
 \cR_i=(q^{*}\cC_i)_{\mathrm{red}},\quad
 q^*\cC_i=d_i\cR_i,\quad
 \aO_{\cY}(\cR_i)\simeq q^*\aH_0.
\end{equation}
The morphism $q$ is \emph{\'etale} over $(\bP^2\times\cU)\setminus\cD$, and we have an isomorphism
\begin{equation}\label{eq:plane-pair-log-etale}
 q^*\Omega^1_{\bP^2\times\cU/\cU}(\log\cD)
 \simeq
 \Omega^1_{\cY/\cU}(\log\cR).
\end{equation}
The relative canonical line bundle
$\omega_{\cY/\cU}:=\det\Omega^1_{\cY/\cU}$ satisfies
\begin{equation}\label{eq:plane-pair-canonical-cover}
 \omega_{\cY/\cU}
 \simeq
 q^*\operatorname{pr}_1^*\aO_{\bP^2}(d_1+d_2-5).
\end{equation}
Consequently, every fiber $Y_a$ has ample canonical line bundle and is of
general type.
\end{prop}

\begin{proof}
Let $y\in\cY$ and let $x=q(y)$. Let $W$ be an open neighborhood of $x$ on which $\aH_0$ is generated by a nowhere-vanishing section $\alpha$, so 
$F_i=f_i\alpha^{d_i}$ for some function $f_i$. We already explained that $q^{-1}(W)$ is the closed subscheme of $W\times\mathbb A^2$ defined by 
$T_1^{d_1}-f_1$ and $T_2^{d_2}-f_2$. 

Suppose that the differential satisfies
\begin{equation}\label{eq:differential-cY}
\lambda_1d(T_1^{d_1}-f_1)(y) +\lambda_2d(T_2^{d_2}-f_2)(y)=0.
\end{equation}
If $f_i(x)\neq 0$, then $T_i(y)\neq 0$, so \eqref{eq:differential-cY} shows that $\lambda_i=0$. The same conclusion holds without hypothesis on $f_i(x)$ when $d_i=1$. Assume now that $f_i(x)=0$, $d_i>1$, and $T_i(y)=0$, so \eqref{eq:differential-cY} becomes a linear equation of $df_i(x)$. They are linearly independent because the two universal curves $\cC_1$ and $\cC_2$ meet transversely in every fiber. Hence $p$ is smooth of relative dimension two. For every $a\in\cU$, restricting \eqref{eq:plane-pair-cover-direct-image} to $\bP^2\times\{a\}$ we get \(\coh^0(Y_a,\aO_{Y_a})=\bC\) because $\coh^0(\bP^2,\aO_{\bP^2}(-m))=0$ for every $m>0$. Thus $Y_a$ is an integral smooth surface. 

On the open subset $q^{-1}(W)$, the sections $q^*F_i$ and $\tau_i$ are represented respectively by $T_i^{d_i}$ and $T_i$. By construction 
$\cR_i=Z(\tau_i)$ is an effective Cartier divisor on $\cY$. We claim that $\cR_i\to\cU$ is smooth. Indeed, 
\[
 \cR_i\cap q^{-1}(W)
 \simeq
 \operatorname{Spec}_W
 \frac{\aO_W[T_j]}
 {(f_i,T_j^{d_j}-f_j)},
\]
while by construction $d(T_j^{d_j}-f_j)$ and $df_i$ are linearly independent again because $\cC_1$ and $\cC_2$ are snc.

On $q^{-1}(W)$, the intersection $\cR_1\cap\cR_2$ is defined by $T_1=T_2=0$. As $T_i^{d_i}=f_i$, we get that $f_1=f_2=0$, and $q$ induces an isomorphism
\[
 q\vert_{\cR_1\cap\cR_2}: \cR_1\cap\cR_2\longrightarrow\cC_1\cap\cC_2.
\]
As $\cC_1+\cC_2$ is snc, $\cR_1\cap\cR_2$ is smooth of relative dimension zero.

Let $y\in\cR_1\cap\cR_2$ and $x=q(y)$. At $x$, the differentials
$df_1,df_2$ form a basis of $\Omega^1_{\bP^2\times\cU/\cU}$. Consequently, the morphism $(f_1,f_2):W\to\bA^2_\cU$ is étale on a neighborhood of $x$.
Locally $\cY$ is obtained from this morphism by the base change \((T_1,T_2)\mapsto(T_1^{d_1},T_2^{d_2})\). So $(T_1,T_2):q^{-1}(W)\to\bA^2_\cU$ is étale on a neighborhood of $y$ while $\cR_i$ is the inverse image of the divisor $T_i=0$. Together with the smoothness of
$\cR_i\to\cU$ proved above, this shows that $\cR=\cR_1+\cR_2$ is snc relative to $\cU$. The formulas in \eqref{eq:plane-pair-branch-pullback} are immediate by construction.

Over $(\bP^2\times\cU)\setminus\cD$, the functions $T_1$ and $T_2$ are
invertible. Since $d_iT_i^{d_i-1}$ is then invertible, we deduce from \eqref{eq:qW-spec-ring} that $q$ is étale over this open subset. 

We next prove \eqref{eq:plane-pair-log-etale}. The morphism $q$ induces
a morphism
\[
 \varphi:
 q^*\Omega^1_{\bP^2\times\cU/\cU}(\log\cD)
 \longrightarrow
 \Omega^1_{\cY/\cU}(\log\cR).
\] 
Let $y\in\cY$ and $x=q(y)$. Suppose first that $x\in\cC_1\cap\cC_2$. The transversality of $\cC_1$ and $\cC_2$ implies that $f_1,f_2$ are local coordinates along the fibers of $\bP^2\times\cU\to\cU$. Similarly, the equations $T_i^{d_i}=f_i$ show that $T_1,T_2$ are local coordinates along the fibers of $\cY\to\cU$. Hence the logarithmic cotangent sheaves have
respectively local bases
\[
 \frac{df_1}{f_1},\frac{df_2}{f_2}
 \quad\text{and}\quad
 \frac{dT_1}{T_1},\frac{dT_2}{T_2}.
\]
With respect to these bases, the morphism $\varphi$ multiplies the two
generators by $d_1$ and $d_2$, respectively. It is therefore an isomorphism near $y$. 

Suppose now that $x\in\cC_i\setminus\cC_j$, where
$\{i,j\}=\{1,2\}$. Pick a local function $g$ such that $f_i,g$ are
local coordinates along the fibers of $\bP^2\times\cU\to\cU$. Since $f_j$ and $T_j$ are invertible near $x$ and $y$, respectively, the equation $T_j^{d_j}=f_j$ defines an étale morphism there. The logarithmic cotangent sheaves have respective local bases
\[
 \frac{df_i}{f_i},dg
 \quad\text{and}\quad
 \frac{dT_i}{T_i},q^*dg.
\]
With respect to these bases, $\varphi$ multiplies the first generator by
$d_i$ and maps the second one to $q^*dg$. Hence $\varphi$ is an
isomorphism near $y$. Outside $\cR$, $\varphi$ is an
isomorphism because $q$ is étale. We have therefore proved \eqref{eq:plane-pair-log-etale}.

Taking determinants in \eqref{eq:plane-pair-log-etale} and using the snc condition, we obtain
\[
 \omega_{\cY/\cU}\otimes\aO_{\cY}(\cR)
 \simeq
 q^*\bigl(\omega_{\bP^2\times\cU/\cU}\otimes\aO_{\bP^2\times\cU}(\cD)\bigr).
\]
Since $\omega_{\bP^2\times\cU/\cU}\simeq\aH_0^{-3}$, $\aO_{\bP^2\times\cU}(\cD)\simeq\aH_0^{d_1+d_2}$, and $\aO_{\cY}(\cR)\simeq q^*\aH_0^2$, we obtain \eqref{eq:plane-pair-canonical-cover}. Note that \eqref{eq:d1d2-condition} implies that $d_1+d_2-5\geq 1$.
\end{proof}

The dual of \eqref{eq:plane-pair-log-etale} gives an isomorphism between $T_{\cY/\cU}(-\log\cR)$ and $q^*T_{\bP^2\times\cU/\cU}(-\log\cD)$, which induces an isomorphism
\begin{equation}\label{eq:projective-bundle-log}
 \bP\bigl(T_{\cY/\cU}(-\log\cR)\bigr)
 \simeq
 \cY\times_{\bP^2\times\cU}
 \bP\bigl(T_{\bP^2\times\cU/\cU}(-\log\cD)\bigr).
\end{equation}
Let
\[
 \widetilde q:
 \bP\bigl(T_{\cY/\cU}(-\log\cR)\bigr)
 \longrightarrow
 \bP\bigl(T_{\bP^2\times\cU/\cU}(-\log\cD)\bigr)
\]
be the second projection in \eqref{eq:projective-bundle-log}. 

Recall that $\cS$ is from Th.~\ref{thm:plane-pair-direction-locus}. Let $\cT:=\widetilde q^{-1}(\cS)$ which is a closed subvariety of $\bP(T_{\cY/\cU}(-\log\cR))$. We denote its fiber over $a\in\cU$ by $T_a$. 

Let $\aH:=q^*\aH_0$, and let $H_a:=\aH\vert_{Y_a}$. Since $q$ is finite,
$\aH$ is $p$-ample. If $C\subset Y_a$ is an integral curve
not contained in $R_a$ with normalization $\mu:\widetilde C\to C$, then the
logarithmic differential
\[
 d\mu: T_{\widetilde C}\bigl(-\log\operatorname{Supp}\mu^*R_a\bigr)
 \to \mu^*T_{Y_a}(-\log R_a)
\]
is nonzero. As in the case of \eqref{eq:log-tangent-lift}, it induces a logarithmic tangent lift morphism
\begin{equation}\label{eq:log-tangent-lift-Y}
\mu_{Y}:\widetilde C \to \bP(T_{Y_a}(-\log R_a)),\quad
x\mapsto \left[d\mu_x\left(T_{\widetilde C,x}
\bigl(-\log\operatorname{Supp}\mu^*R_a\bigr)\right)\right].
\end{equation}
We will call the image of $\mu_Y$ the logarithmic tangent lift of $C$.

\begin{lem}\label{lem:Y-lift}
Let $a\in\cU$, and let $\Gamma\not\subset D_a$ be a rational
curve with normalization $\nu:\widetilde\Gamma\to\Gamma$. Assume that $\#\operatorname{Supp}\nu^*D_a\leq 2$. Then there is a rational
curve $C\subset Y_a$, not contained in $R_a$, such that $q_a(C)=\Gamma$
and the logarithmic tangent lift of $C$ is contained in $T_a$. If $\delta$ is the degree of $q_a\vert_C:C\to\Gamma$, then \begin{equation}\label{eq:plane-pair-lift-degree}
 H_a\cdot C=\delta\deg\Gamma.
\end{equation}
\end{lem}

\begin{proof}
Let $Z$ be an irreducible component of $Y_a\times_{\bP^2}\widetilde\Gamma$. It dominates $\widetilde\Gamma$. Let $\widetilde Z$ be the normalization of $Z$, and let
$\varphi:\widetilde Z\to\widetilde\Gamma$ be the induced finite morphism. Since $q_a$ is étale over $\bP^2\setminus D_a$, the morphism
$\varphi$ is étale outside $\operatorname{Supp}\nu^*D_a$.

The total ramification over each point of $\widetilde\Gamma\simeq\bP^1$ is at most $\delta-1$. By the Riemann--Hurwitz formula, we have
\[
 2g(\widetilde Z)-2
 \leq -2\delta+2(\delta-1)=-2.
\]
Hence $\widetilde Z\simeq\bP^1$.

Let $C\subset Y_a$ be the image of $Z$. Over the open subset of $\Gamma$
on which $\nu$ is an isomorphism, the projection $Z\to C$ is an
isomorphism onto its image. Thus $\widetilde Z\to C$ is the
normalization of $C$, and $C$ is rational. By construction $q_a(C)=\Gamma$ and $C\not\subset R_a$.

At a general point of $C$, the differential of $q_a$ maps
the tangent line of $C$ to the tangent line of $\Gamma$. Thus
$\widetilde q$ maps a dense open subset of the logarithmic tangent lift
of $C$ into the logarithmic tangent lift of $\Gamma$. The latter is
contained in $S_a$ by Theorem~\ref{thm:plane-pair-direction-locus}.
Since $T_a=\widetilde q^{-1}(S_a)$ is closed, the logarithmic tangent
lift of $C$ is contained in $T_a$. Finally \eqref{eq:plane-pair-lift-degree} follows from the projection formula because $(q_{a})_*[C]=\delta[\Gamma]$.
\end{proof}

\subsection{Symmetric differential and degree bound}\label{subsec:log-symmetric-differential-degree-bound}

\begin{prop}\label{prop:plane-pair-relative-symmetric-differential}
Up to replacing $\cU$ by a nonempty Zariski-open subset there are integers $N,b\geq 1$ with the following properties. There is a nonzero section
\begin{equation}\label{eq:plane-pair-symmetric-differential}
 \boldsymbol\omega\in \coh^0\!\left(\cY, \Sym^N\Omega^1_{\cY/\cU} \otimes(\aH)^{\otimes bN}
 \otimes\aO_{\cY}(N\cR)\right)
\end{equation}
whose restriction $\omega_a$ to $Y_a$ is nonzero for every
$a\in\cU$. If $C\subset Y_a$ is an integral curve not contained
in $R_a$ and its logarithmic tangent lift is contained in $T_a$, then
$C\subset Z(\omega_a)$ or $C$ is $\omega_a$-invariant.
\end{prop}

\begin{proof}
Denote $\mathscr P:=\bP(T_{\cY/\cU}(-\log\cR))$. Let $\pi:\mathscr P\to\cY$ be the projection, and let $\rho:=p\circ\pi$. Pick $b\geq 1$ such that $\mathcal B:=\aO_{\mathscr P}(1)\otimes\pi^*\aH^{\otimes b}$ is ample relative to $\rho$. Let $\mathcal I_{\cT}$ be the ideal sheaf of $\cT$
in $\mathscr P$. By Serre's theorem, for $N\gg 0$ and up to shrinking $\cU$, we can find a section $\boldsymbol\beta$ of $\mathcal I_{\cT}\otimes\mathcal B^{\otimes N}$ whose restriction to $\mathscr P_a$ is nonzero for every $a\in\cU$.

We have
\(\pi_*\aO_{\mathscr P}(N)\simeq\Sym^N\Omega^1_{\cY/\cU}(\log\cR)\). Thus projecting $\boldsymbol\beta$ onto $\cY$ we get a nonzero section
\[
 \boldsymbol\theta\in \coh^0
 \left(\cY, \Sym^N\Omega^1_{\cY/\cU}(\log\cR)
 \otimes(\aH)^{\otimes bN}\right).
\]
For $a\in\cU$ we denote the restriction of $\boldsymbol\theta$ to $Y_a$ by $\theta_a$. By construction $\theta_a(v^{\otimes N})=0$ for every direction $[v]\in T_a$.

Logarithmic one-forms have at most simple poles along $\cR$.
Hence there is a natural inclusion
\begin{equation}\label{eq:log-differential-Y-inclusion}
\Omega^1_{\cY/\cU}(\log\cR) \longrightarrow \Omega^1_{\cY/\cU}\otimes\aO_{\cY}(\cR)
\end{equation}
whose $N$-th symmetric power maps $\boldsymbol\theta$ to the desired section $\boldsymbol\omega$ in \eqref{eq:plane-pair-symmetric-differential}. The morphism \eqref{eq:log-differential-Y-inclusion} is an
isomorphism over $\cY\setminus\cR$, so $\omega_a$ is also nonzero for every
$a\in\cU$.

Let $C\subset Y_a$ be an integral curve not contained in $R_a$, and
assume that its logarithmic tangent lift is contained in $T_a$. At a
general smooth point $y\in C\setminus R_a$, the logarithmic tangent
sheaf coincides with the ordinary tangent sheaf. The line
$T_{C,y}\subset T_{Y_a,y}$ corresponds to a point of $T_a$, and
$\theta_a(y)$ vanishes on $\Sym^N T_{C,y}$. Since \eqref{eq:log-differential-Y-inclusion} is an isomorphism near $y$, the differential $\omega_a(y)$ also vanishes on $\Sym^N T_{C,y}$. If $C\not\subset Z(\omega_a)$, then the curve $C$ is $\omega_a$-invariant.
\end{proof}

\begin{prop}\label{prop:plane-pair-uniform-degree}
Up to replacing $\cU$ by a nonempty Zariski-open subset, there exists an integer $E\geq 1$ such that every rational curve $\Gamma\not\subset D_a$, where $a\in\cU$, satisfies $\deg\Gamma\leq E$ whenever its normalization
$\nu:\bP^1\to\Gamma$ satisfies $\#\operatorname{Supp}\nu^*D_a\leq 2$.
\end{prop}

\begin{proof}
Recall that the morphism $\cY\to\cU$ is a smooth projective family of smooth surfaces of general type, and $\aH$ is relatively ample; see Proposition~\ref{prop:plane-pair-cyclic-cover}. We apply
Theorem~\ref{thm:main-degree-bound} to the section \eqref{eq:plane-pair-symmetric-differential}. Up to shrinking $\cU$ we obtain an integer $E\geq 1$ such that $H_a\cdot C\leq E$ for every rational curve $C\subset Y_a$ which is contained in $Z(\omega_a)$ or is $\omega_a$-invariant.

Let $\Gamma$ satisfy the assumptions of the proposition. Let $C\subset Y_a$ be a curve satisfying the conclusions of Lemma~\ref{lem:Y-lift}.
By Proposition~\ref{prop:plane-pair-relative-symmetric-differential},
either $C\subset Z(\omega_a)$ or $C$ is $\omega_a$-invariant. Thus
$H_a\cdot C\leq E$. Then the equality \eqref{eq:plane-pair-lift-degree} from Lemma~\ref{lem:Y-lift} implies $\deg\Gamma\leq E$.
\end{proof}

\subsection{The Zariski-open conclusion}

\begin{lem}\label{lem:plane-pair-bounded-incidence}
For every integer $E\geq1$, the locus of parameters $a\in\cU$ for which
there is a rational curve $\Gamma\not\subset D_a$ of degree at most $E$
whose normalization $\nu:\bP^1\to\Gamma$ satisfies
$\#\operatorname{Supp}\nu^*D_a\leq2$ is constructible.
\end{lem}

\begin{proof}
Fix an integer $e$ with $1\leq e\leq E$. Let
$\operatorname{Mor}_e(\bP^1,\bP^2)$ be the quasi-projective variety parametrizing morphisms $f:\bP^1\to\bP^2$ such that
$f^*\aO_{\bP^2}(1)\simeq\aO_{\bP^1}(e)$. Let $\mathscr M_e$ be the open
subset of $\operatorname{Mor}_e(\bP^1,\bP^2)\times\cU$ consisting of
pairs $(f,a)$ such that $f(\bP^1)$ is contained in neither component of
$D_a$. We define a morphism
\[
 \Phi_e:\mathscr M_e\longrightarrow
 \lvert\aO_{\bP^1}(d_1e)\rvert\times\lvert\aO_{\bP^1}(d_2e)\rvert
\]
which to $f$ associates the pullback of the two components of $D_a$. 
Let $Z_e\subset \lvert\aO_{\bP^1}(d_1e)\rvert\times\lvert\aO_{\bP^1}(d_2e)\rvert$ be the locus of pairs of effective divisors whose union of
supports has cardinality at most two. This locus is closed. Hence the projection of $\Phi_e^{-1}(Z_e)$ in $\cU$ is constructible.
Finally it suffices to observe that the finite union of the projections of $\Phi_e^{-1}(Z_e)$ over $1\leq e\leq E$ is the desired locus in the statement.
\end{proof}

For a very general parameter $a\in\cU$, Xi Chen's theorem \cite[Cor.~1.19]{Chenlog} asserts that every integral curve $\Gamma\not\subset D_a$, with normalization $\nu:\widetilde\Gamma\to\Gamma$, satisfies
\begin{equation}\label{eq:plane-pair-chen}
2g(\widetilde\Gamma)-2+\#\operatorname{Supp}\nu^*D_a\geq(d_1+d_2-4)\deg\Gamma.
\end{equation}
If $\Gamma$ is rational and $\#\operatorname{Supp}\nu^*D_a\leq2$, the left-hand side is nonpositive, whereas the right-hand side is positive by
\eqref{eq:d1d2-condition}. Hence no such curve exists for a very general
parameter in $\cU$.

\begin{proof}[Proof of Theorem~\ref{mainthm:plane-pair-complement}]
We shrink $\cU$ if necessary so the conclusions of Proposition~\ref{prop:plane-pair-uniform-degree} and Lemma~\ref{lem:plane-pair-bounded-incidence} hold. We fix $E$ the integer in Proposition~\ref{prop:plane-pair-uniform-degree}, and $\Sigma\subset\cU$ the constructible locus in Lemma~\ref{lem:plane-pair-bounded-incidence}. The subset $\Sigma$ is not dense in $\cU$ because otherwise it would contain a very general parameter which contradicts Xi Chen's inequality \eqref{eq:plane-pair-chen}. Thus $\overline\Sigma$ is a proper closed
subset of $\cU$.

Let $a\in\cU\setminus\overline\Sigma$. If there were a rational curve
$\Gamma\not\subset D_a$ whose normalization satisfied
$\#\operatorname{Supp}\nu^*D_a\leq2$, then, by
Proposition~\ref{prop:plane-pair-uniform-degree}, we would have
$\deg\Gamma\leq E$. Therefore $a\in\Sigma$ by the definition of $\Sigma$, a contradiction.
\end{proof}

\subsection{Hyperbolic embedding}\label{subsec:hyperbolic-embedding} 

\begin{proof}[Proof of Theorem~\ref{mainthm:plane-complement-hyperbolicity}]
The nonempty Zariski-open subset obtained in Theorem~\ref{mainthm:plane-pair-complement} meets the nonempty Zariski-open subset supplied by Hou--Wang--Xie's theorem \cite[Th.~1.1]{HouWangXie2026} (the locus where their Nevanlinna Second Main Theorem holds), since $\cP_{d_1,d_2}^{\mathrm{snc}}$ is irreducible. Let $a$ belong to their intersection, and let $f:\bC\to\bP^2\setminus D_a$ be a nonconstant entire curve. The Second Main Theorem of \cite[Th.~1.1]{HouWangXie2026} implies that $f$ is necessarily algebraically degenerate, i.e.\ has image contained in a curve. 

Let $\Gamma$ be the Zariski closure of $f(\bC)$, and let $\nu:\widetilde\Gamma\to\Gamma$ be its normalization. The map $f$ lifts
to a holomorphic map
\[
 \widetilde f:\bC\longrightarrow
 \widetilde\Gamma\setminus\operatorname{Supp}\nu^*D_a.
\]
If $g(\widetilde\Gamma)\geq 2$, the target is of course hyperbolic. If
$g(\widetilde\Gamma)=1$, then $\deg\nu^*D_a=(d_1+d_2)\deg\Gamma>0$, so at least one point is removed. If $g(\widetilde\Gamma)=0$,  then at least three points are removed by Theorem~\ref{mainthm:plane-pair-complement}. The target is therefore hyperbolic in every case, and $\widetilde f$ must be constant. This contradiction proves that $\bP^2\setminus D_a$ is Brody hyperbolic.

For $i,j=1,2$, by adjunction and B\'ezout's theorem we have
\[
 2g(C_{i,a})-2+\#(C_{i,a}\cap C_{j,a})
 =d_i(d_1+d_2-3)>0.
\]
Thus $C_{i,a}\setminus C_{j,a}$ is hyperbolic. We may therefore apply
Green's theorem~\cite{GreenLog} (cf.\ \cite[Th.~3.6.13]{Kobayashi1998}) and deduce that $\bP^2\setminus D_a$ is hyperbolically embedded in $\bP^2$. In particular, it is Kobayashi hyperbolic.
\end{proof}

\section*{Acknowledgements}

We thank Professor Yum-Tong Siu for a key suggestion made to S.-Y. Xie during the 2024 program ``Progress in Several Complex Variables and Complex Geometry'' at the Tianyuan Mathematics Research Center in Kunming: to use the direct-image theorem to obtain the required Zariski-open subset of the parameter space. We thank the Center for providing an excellent environment for mathematical exchange during our multiple visits. S.-Y. Xie presented this problem during a discussion session at the workshop ``Complex Hyperbolicity, Function Fields and Non-Archimedean Arithmetic'' at the Institute of Mathematics, Academia Sinica, Taipei, in June 2026. He thanks Julie Tzu-Yueh Wang for organizing the workshop, the participants for their discussions, and Academia Sinica for its hospitality. He first presented the main result publicly at the 32nd ICFIDCAA in Quanzhou on 22 August 2026 and thanks the organizers.

\section*{Funding}
S.-Y. Xie acknowledges partial support from the National Key R\&D Program of China under Grants No. 2023YFA1010500 and No. 2021YFA1003100, and from the National Natural Science Foundation of China under Grants No. 12288201 and No. 12471081, as well as support from the  Xiaomi Young Talents Program.

S.-Y. Zhao acknowledges partial support from the French National Research Agency under the projects GAG (ANR-24-CE40-3526-01) and DynAtrois (ANR-24-CE40-1163).

\section*{AI usage}
The strategy of relating the Zariski-open condition to Poincaré's problem is purely developed by the authors, initially inspired by the above mentioned insight of Siu. The bulk of this paper was done before the use of AI. ChatGPT5.6 helped us to find Lu--Miyaoka's paper, which we were not aware of and which simplified our old approaches. Apart from this, AI (ChatGPT5.6 and Deepseek) were also used for editorial work of the manuscript. 

\bibliographystyle{amsplain}
\bibliography{references0903}

\end{document}